\documentclass[11pt]{article}
\usepackage{amsmath,graphicx,natbib,url}
\usepackage{enumitem}
\usepackage{amssymb,amsthm}
\usepackage{empheq}
\usepackage{dsfont}
\usepackage{mathrsfs}
\usepackage{graphicx}
\usepackage{float}
\usepackage{subcaption}
\usepackage{bbm,bm} 
\usepackage{booktabs}
\usepackage{bigints} 
\usepackage{fancyhdr}
\usepackage[colorlinks = true, 
linkcolor = blue, 
urlcolor  = blue, 
citecolor = blue,
anchorcolor = blue]{hyperref}
\usepackage{algorithm,tabularx,amsfonts}
\usepackage{algpseudocode}
\usepackage{authblk}
\usepackage{multirow}
\usepackage{ulem}
\usepackage{xcolor}
\definecolor{darkgreen}{RGB}{0,120,0}

\newcommand{\blind}{0}

\makeatletter                             
\newcommand\mysection{\@startsection {section}{1}{\z@}{-3.5ex \@plus -1ex \@minus -.2ex}{2.3ex \@plus.2ex}
{\normalfont\large\bfseries}}
\makeatother    
\newcommand{\vect}[1]{\boldsymbol{#1}}

\newcommand{\inner}[1]{\left\langle #1 \right\rangle}

\newcommand{\phj}{\phi_{j\vect{\gamma}}}
\newcommand{\phs}{\phi_{s\vect{\gamma}}}
\newcommand{\tphj}{\widetilde{\phi}_{j\vect{\gamma}}}

\newcommand{\tphs}{\widetilde{\phi}_{s\vect{\gamma}}}
\newcommand{\hj}{h_{j\vect{\beta}}}

\newcommand{\thj}{\widetilde{h}_{j\vect{\beta}}}
\newcommand{\thjn}{\widetilde{h}_{j\widehat{\vect{\beta}}_n}}
\newcommand{\ths}{\widetilde{h}_{s\vect{\beta}}}

\newcommand{\rtcru}{RT\,Cru}

\begin{document}

\def\spacingset#1{\renewcommand{\baselinestretch}%
{#1}\small\normalsize} \spacingset{1}


\if0\blind
{
  {\title{\bf A New Class of Asymptotically Distribution-Free Smooth Tests}}
\author[1]{Xiangyu Zhang}
\author[1]{Sara Algeri}
\affil[1]{\footnotesize{School of Statistics, University of Minnesota}
\vspace{3mm}
}
\affil[ ]{224 Church Street SE
Minneapolis, MN 55455, USA. 
\vspace{3mm}}
\affil[ ]{{\textit{Emails: \href{mailto:zhan6004@umn.edu}{zhan6004@umn.com}}}; {\textit{ \href{mailto:salgeri@umn.edu}{salgeri@umn.edu}}}}
\date{}
  \maketitle
} \fi

\if1\blind
{
  \bigskip
  \bigskip
  \bigskip
  \begin{center}
    {\LARGE\bf Title}
\end{center}
  \medskip
} \fi
\bigskip
\bigskip
\begin{abstract}
This article demonstrates how recent developments in the theory of empirical processes allow us to construct a new family of asymptotically distribution-free smooth tests. Their distribution-free property is preserved even when the parameters are estimated, model selection is performed, and the sample size is only moderately large.  A computationally efficient alternative to the classical parametric bootstrap is also discussed. 
\end{abstract}

\noindent%
\vfill

\newpage
\spacingset{1.3} 
\section{Introduction} \label{sec:1}

Let $X$ be a continuous random variable with cumulative distribution function (CDF) $Q$ and probability density function (PDF) $q$
with support $\mathcal{X}\subseteq \mathbb{R}$. We are interested in assessing whether the unknown distribution $Q$ belongs to a family of distribution functions $G_{\vect{\beta}}$, with PDF $g_{\vect{\beta}}$,  differentiable with respect to $\vect{\beta} \in \mathcal{B} \subseteq \mathbb{R}^p$. To tackle this problem, we may consider an alternative model of the form 
\begin{equation}
    \begin{aligned}
 g_{\vect{\beta}}(x)\Big\{1+ \sum_{j=1}^{m} \theta_{j\vect{\beta}} h_{j\vect{\beta}}(x)\Big\}
\label{eqn:alter}
    \end{aligned}
\end{equation}
which incorporates the null density, $g_{\vect{\beta}}$, as a special case. The functions $\{h_{j\vect{\beta}}\}_{j=1}^\infty$ form an orthonormal basis in $ L^2(G_{\vect{\beta}})$; hence 
$$\int_{\mathcal{X}} h_{i\vect{\beta}}(x) h_{j\vect{\beta}}(x) d G_{\vect{\beta}}(x)=\mathds{1}_{\{i=j\}},$$
where $\mathds{1}_{\{\cdot\}}$ denotes the indicator function.
Assuming that the density ratio $\frac{q}{g_{\vect{\beta}}}\in L^2(G_{\vect{\beta}})$, as $m\rightarrow\infty$, the term in the curly brackets in \eqref{eqn:alter} provides an orthonormal expansion for $\frac{q}{g_{\vect{\beta}}}$. Hence, the coefficients 
of such an expansion correspond to $$\theta_{j\vect{\beta}}= \int_{\mathcal{X}} h_{j\vect{\beta}}(x)\frac{q(x)}{g_{\vect{\beta}}(x)} dG_{\vect{\beta}}(x)=\int_{\mathcal{X}} h_{j\vect{\beta}}(x)dQ(x),\quad \text{for all $j=1,\dots,m$}.$$ 

A smooth test aims to assess the validity of $g_{\vect{\beta}}$ by testing the null hypothesis:
\begin{equation}
    \begin{aligned}
   &H_0:  \theta_{1\vect{\beta}}=...=\theta_{m\vect{\beta}}=0, \text{ for some }\vect{\beta}\in \mathcal{B} 
\label{eqn:h0h1}
    \end{aligned}
\end{equation}
against smooth alternatives of the form \eqref{eqn:alter}, where at least one of the coefficients $\theta_{1\vect{\beta}},\ldots,$ $\theta_{m\vect{\beta}}$ is nonzero.
One could also rely on an exponential tilt to express the deviations of the alternative from $g_{\vect{\beta}}$ through the orthonormal set $\{h_{j\vect{\beta}}\}_{j=1}^m$, as in the original formulation of smooth tests proposed by \citet{neyman1937}. Other formulations may involve normalizing the expansion in \eqref{eqn:alter} to ensure that the resulting density is non-negative \citep[e.g.,][]{gajek1986}. 
Here, we rely on the alternative in \eqref{eqn:alter},  proposed by \cite{Barton1953}, by virtue of its simplicity, and we implicitly assume that the $\theta_{j\vect{\beta}}$ coefficients are suitably constrained to ensure \eqref{eqn:alter} is non-negative\footnote{For example,  as discussed in \cite{Barton1953}, when $h_{j\vect{\beta}}$ are the normalized shifted Legendre polynomials evaluated at the probability integral transform $G_{\vect{\beta}}(x)$ and $m=1$, non-negativity of \eqref{eqn:alter} is achieved whenever $|\theta_{1\vect{\beta}}|<\frac{1}{3}$; whereas, if $m=2$, we must have $|\sqrt{5}\theta_{2\vect{\beta}}|\geq |\sqrt{3}\theta_{1\vect{\beta}}|$.}. 
Nevertheless, the methods described in what follows can be adapted to other specifications of the alternative model for which the testing problem can be formulated as in \eqref{eqn:h0h1}.

Several test statistics have been proposed in the literature for testing \eqref{eqn:h0h1}. Prominent examples include the score test statistic \citep[][]{neyman1937, Barton1953}, which consists of the sum of squares of the estimated coefficients in the expansion in \eqref{eqn:alter}, and typically employed for testing simple hypotheses; the generalized score test statistic\footnote{Also known as the generalized smooth test statistic or efficient score test statistic.} \citep[][]{javitz1975generalized, david1979} used for testing composite hypotheses; and the order selection test statistic \citep[][]{marc1999} which naturally accounts for the selection of $m$ in \eqref{eqn:alter}. Their asymptotic null distributions are known, but may require a very large sample size to reach such limits \citep{Klar2000}. 
Data-driven smooth tests \citep{Led1994} focus on the situation in which the basis functions to be included in the expansion in \eqref{eqn:alter} are determined by data-driven selection criteria -- such as BIC \citep[e.g.,][]{Led1994,kall1995a,kall1997jasa, Inglot1997,inglot2006atowards}. 
Under the null model $G_{\vect{\beta}}$, classical smooth test statistics are asymptotically distribution-free.

For all the above-mentioned tests, asymptotic distribution-freeness arises as a consequence of the inherent structure of the test statistics, while the choice of basis $\{h_{j\vect{\beta}}\}_{j=1}^\infty$ remains somewhat arbitrary. We propose a shift of perspective in which a careful construction of the basis ensures that any test statistic based on its elements is asymptotically distribution-free by design. Such a construction relies on the \textit{Khmaladze-2 (K2) transform}\footnote{The name
``Khmaladze-2 transform'' is used to distinguish the transformation employed in this manuscript and proposed in \citet{K22016} from an earlier ``Khmaladze transform'' introduced by the same author in \cite{gofk1}.}  introduced by \cite{K22016}, which enables both accurate and computationally efficient inference. 

The K2 transform uses unitary operators to map model-dependent projections into a ``standard projection'' that is arbitrarily chosen by the user. This standard projection does not vary with the model being tested and is asymptotically distribution-free under the null hypothesis.
In just over a decade, the K2 transformation has enabled several advances in the goodness-of-fit literature across various contexts---from testing discrete and multivariate parametric distributions \citep{khm13,K22016} to regression \citep{khm21,KHMALADZE2017348} and grouped data analysis \citep{sara2024}. By enabling its use in the context of smooth tests, this work extends its usefulness beyond goodness-of-fit. Specifically, it provides applied scientists with a new suite of distribution-free tests that offer a middle ground between goodness-of-fit tests and directional tests such as Neyman-Pearson.

While asymptotically distribution-freeness is especially desirable when testing complex distributions and/or when computationally intensive parameter estimation is needed \textemdash \ as is often the case in physics and astronomy \citep[e.g.,][]{GAMBIT1,Cusin:2017mjm,Cusin:2019jpv, PRD,PRDL} \textemdash \ in other settings, applied scientists may choose to derive the null distributions of the test statistics by means of resampling methods. In Section~\ref{sec:3}, we discuss a computationally efficient alternative to the classical parametric bootstrap, called the \textit{projected bootstrap}. Specifically, one can avoid re-estimating the unknown parameter $\vect{\beta}$ on each bootstrap replicate by exploiting a projection of the underlying empirical process. The efficacy of this approach when deriving the distribution of classical goodness-of-fit statistics estimated via maximum likelihood has been investigated numerically by \cite{algeri2022k}. Here, a formal proof is provided to confirm the validity of this approach in the context of smooth tests and for general $Z$-estimators. 

\section{Smooth tests and the function-parametric empirical process} \label{sec:2}

Let $X_1, \ldots, X_n$ be independent and identically distributed (IID) random variables with CDF $Q$. Consider the Hilbert space $L^2(G_{\vect{\beta}}) = \{ h: \inner{h,h}_{G_{\vect{\beta}}}<\infty\},$
with inner product: 
$$\inner{h,h'}_{G_{\vect{\beta}}} = \int_{\mathcal{X}}h(x)h'(x)dG_{\vect{\beta}}(x).$$
The \textit{function-parametric} empirical process $v_{\vect{\beta},n}(h)$ indexed by function $h$ in $L^2(G_{\vect{\beta}})$ is defined as: 
\begin{equation*}
    \begin{aligned}
    &v_{\vect{\beta},n}(h) = \int_{\mathcal{X}}h(x)dv_{\vect{\beta},n}(x) = 
    \frac{1}{\sqrt{n}}\sum_{i=1}^n \bigl[h (X_i)- \mathbb{E}_{G_{\vect{{\beta}}}}[h (X_i)]
\bigl]
    \label{eqn:2dv}
    \end{aligned}
\end{equation*}
and 
$$v_{\vect{\beta},n}(x) = \frac{1}{\sqrt{n}} \sum\limits_{i=1}^n \left[\mathds{1}_{\left\{X_{i} \leq x\right\}}-G_{\vect{\beta}}(x)\right]$$
is the classical empirical process. Hence, $v_{\vect{\beta},n}(h) = v_{\vect{\beta},n}(x)$ when $h(z)=\mathds{1}_{\{z\leq x\}}$. 

Most statistics proposed in the literature to test \eqref{eqn:h0h1} can be specified as functionals of the process $v_{\vect{\beta},n}(h)$ and its projections. In particular, let 
$$\bm{u}_{\vect{\beta}}=\nabla_{\vect{\beta}}\ln g_{\vect{\beta}}  \quad \text{and} \quad  \Gamma_{\boldsymbol{\beta}}=\inner{\vect{u}_{\vect{\beta}},\vect{u}^T_{\vect{\beta}}}_{G_{\beta}}$$ be, respectively, the score function and the Fisher information matrix of $G_{\vect{\beta}}$. Define the orthonormalized score function, with coordinates in $L^2(G_{\vect{\beta}})$,  as 
\begin{equation*}
    \begin{aligned}
\vect{b}_{\vect{\beta}} = \left[b_{\vect{\beta}_1}, \ldots, b_{\vect{\beta}_p}\right]^T=\Gamma_{\vect{\beta}}^{-1 / 2} \vect{u}_{\vect{\beta}},
    \end{aligned}
\end{equation*} 
where $\Gamma_{\vect{\beta}}^{-1 / 2}$ denotes the principal square root matrix of $\Gamma_{\vect{\beta}}^{-1}$. 
Denote with $\{\widetilde{h}_{j\vect{\beta}}\}_{j=1}^m$ the ``residuals'' of $\{h_{j\vect{\beta}}\}_{j=1}^m$ in \eqref{eqn:alter} after an orthogonal projection onto $\vect{b}_{\vect{\beta}}$, i.e., 
\begin{equation}
    \begin{aligned}
    \widetilde{h}_{j\vect{\beta}} =  h_{j\vect{\beta}} -\vect{b}^T_{\vect{\beta}}\left\langle \vect{b}_{\vect{\beta}},h_{j\vect{\beta}}\right\rangle_{G_{\vect{\beta}}}
    = h_{j\vect{\beta}} - \sum_{k=1}^p b_{\vect{\beta}_k}\left\langle h_{j\vect{\beta}},b_{\vect{\beta}_k}\right\rangle_{G_{\vect{\beta}}},  \text{ for } j=1,\ldots,m.
    \label{eqn:tildeh}
    \end{aligned}
\end{equation}
Deviations within the parametric family $\{G_{\vect{\beta}}, \ \vect{\beta}\in \mathcal{B}\}$ occur in the direction given by the the score function $\vect{u}_{\vect{\beta}}$. Therefore, since the functions  $\widetilde{h}_{j\vect{\beta}}$ are obtained by projecting $h_{j\vect{\beta}}$ onto the orthogonal complement of the tangent space given by the span of $\vect{u}_{\vect{\beta}}$, they capture deviations outside the hypothesized parametric family. 

The most widely used statistics for testing $\eqref{eqn:h0h1}$ is the generalized score test statistic, \citep[cfr.][]{javitz1975generalized,david1979}, defined as 
\begin{equation}
    \begin{aligned}
\widehat{S}_{m,n}  
= \sum_{i=1}^m \sum_{j=1}^m v_{\widehat{\vect{\beta}}_n,n}(\widetilde{h}_{i\widehat{\vect{\beta}}_n})\bigl(\Sigma^{-1}_{m,\widehat{\vect{\beta}}_n}\bigl)_{ij} v_{\widehat{\vect{\beta}}_n,n}(\widetilde{h}_{j\widehat{\vect{\beta}}_n})
\label{eqn:gst}
    \end{aligned}
\end{equation}  
where $\Sigma^{-1}_{m,\widehat{\vect{\beta}}_n}$ is the inverse of estimated variance-covariance matrix, ${\Sigma}_{m,\widehat{\vect{\beta}}_n}$, of elements
$$\bigl({\Sigma}_{m,\widehat{\vect{\beta}}_n}\bigl)_{ij} = \mathbb{E}_{G_{\vect{{\beta}}}}\left[v_{\vect{\beta},n}(\widetilde{h}_{i{\vect{\beta}}})v_{\vect{\beta},n}(\widetilde{h}_{j{\vect{\beta}}})\right]\Bigl|_{\vect{\beta}=\vect{\widehat{\beta}}_n},$$ 
which we assume to be a continuous function in $\vect{\beta}$, and 
$\widehat{\vect{\beta}}_n$ denotes a consistent estimator of $\vect{\beta}$ (see condition~\ref{A1} in Section~\ref{sec:3}). 

When $m$, the number of basis functions to be used in \eqref{eqn:alter} and \eqref{eqn:gst}, is determined on the basis of the data observed, the so-called order selection test statistic \citep[cfr.][]{marc1999} can be used to incorporate the choice of the order $m$ directly into its formulation.
Specifically, let $m$ be the maximizer of the selection criterion
\begin{equation}
\widehat{S}_{m,n} - C_{\alpha,n} m, \quad \text{ with } m=0,\dots, M_n,
\label{eqn:selection_criteria}
\end{equation}
in which $\widehat{S}_{m,n}$ is the generalized score test statistic in \eqref{eqn:gst} or its non-normalized counterpart:
\begin{equation}
\widehat{S}_{m,n} = \sum_{j=1}^m v^2_{\widehat{\vect{\beta}}_n,n}(\widetilde{h}_{j\widehat{\vect{\beta}}_n}).
\label{eqn:non_norm}
\end{equation}
Let $M_n$ denote the maximum number of basis functions to be considered, which could either be fixed or grow to infinity as $n\to\infty.$ Selecting   $m$ using the criterion in \eqref{eqn:selection_criteria} corresponds to choosing the first $m$ basis functions among $\{h_{j\vect{\beta}}\}_{j=1}^{M_n}$ to be included in the expansion in \eqref{eqn:alter}. 

The underlying idea of a test based on \eqref{eqn:selection_criteria} is to reject the null hypothesis in \eqref{eqn:h0h1} if the criterion in \eqref{eqn:selection_criteria} is larger than zero for some $m$ in $\mathcal{M}_n = \left\{1, \ldots, M_n\right\}$ -- that is, when an alternative model of the form in \eqref{eqn:alter} is favored over $G_{\vect{\beta}}$. 
Equivalently, $G_{\vect{\beta}}$ is rejected when  the \textit{order selection test statistic} defined as
\begin{equation}
\begin{aligned}
\widehat{T}_n=\max _{m \in \mathcal{M}_n}\left\{\frac{\widehat{S}_{m,n}}{m}  \right\},
\label{eqn:ordersel}
\end{aligned}
\end{equation}
is such that $\widehat{T}_n \geq C_{\alpha,n}$, for some constant $C_{\alpha,n}$ that controls the significance level $\alpha$ of the test. 

The criterion in \eqref{eqn:selection_criteria} can be generalized so that basis functions with indexes $j$ in any subset $B$ of $\{1,\dots, M_n\}$, not necessarily ordered from $j = 1, \dots, m$, are selected \citep[][pp.103-107]{thas2010}. In this case, the set of indices $B$ is chosen to maximize
\begin{equation}
\widehat{S}_{B,n} - C_n |B|, \quad \text{ with } B \subseteq \mathcal{M}_n, 
\label{eqn:selection_criteria_subset}
\end{equation}
where $|B|$ denotes the cardinality of $B$ and $\widehat{S}_{B,n}$ is the counterpart of $\widehat{S}_{m,n}$ with the summations in \eqref{eqn:gst} and \eqref{eqn:non_norm} taken over the indexes in $B$. 
The corresponding \textit{subset selection test statistic}  specifies as 
\begin{equation}
\begin{aligned}
    \widetilde{T}_{n} = \max_{ B\subseteq \mathcal{M}_n: B \neq \emptyset} \left\{\frac{\widehat{S}_{B, n}}{|B|} \right\}.
\label{eqn:subsetsel}
\end{aligned}
\end{equation}
Compared to $\widehat{T}_n$ in \eqref{eqn:ordersel}, $\widetilde{T}_n$ offers greater flexibility in the selection of the basis functions and thus accommodates a wider range of possible alternative models to be employed in the expansion in \eqref{eqn:alter}. 

The limiting distribution of the subset selection statistic in $\eqref{eqn:subsetsel}$ cannot be easily derived. Therefore, a computationally efficient algorithm is required to simulate its null distribution, especially when $M_n$ is large. In the next section, we introduce an alternative to the classical parametric bootstrap to perform this task. 

\section{Smooth tests via projected bootstrap} \label{sec:3}
The classical parametric bootstrap requires re-estimating the unknown parameter $\vect{\beta}$ on each bootstrap sample to account for the variability introduced by estimation. In some instances, however, a repeated estimation can make the procedure computationally intensive. The so-called ``projected bootstrap" \citep{algeri2022k,sara2024} overcomes this limitation by exploiting the projection structure induced by parameter estimation 
 to avoid repeating the estimation of the parameter on each bootstrap replicate. As shown in what follows, such a projection arises in the context of smooth tests rather organically. 

Let the true (but unknown) parameter vector under the null be $\vect{\beta}_0$, that is, if
$H_0$ in \eqref{eqn:h0h1} holds, then $Q\equiv G_{\vect{\beta}_0}$. Let $\mathcal{N}\subset\mathcal{B}$ be the closure of a neighborhood of the true parameter value $\vect{\beta}_0$.  Denote with $\mathcal{L}\left(G_{\vect{\beta}}\right)$ the subspace of $L^2\left(G_{\vect{\beta}}\right)$ given by
$$\mathcal{L}\left(G_{\vect{\beta}}\right)=\left\{ h \in L^2\left(G_{\vect{\beta}}\right): \inner{h,\bm{1}}_{G_{\vect{\beta}}}=0\right\}.$$ 
In what follows, we focus our attention on the subset of functions in $\mathcal{L}\left(G_{\vect{\beta}}\right)$ that depend on $\vect{\beta}$, with $\vect{\beta}\in\mathcal{N}$. Such functions primarily intervene in our study by either characterizing the alternative in \eqref{eqn:alter} or defining the estimating equations for $\vect{\beta}$. 
We assume such functions to be continuously differentiable in $\vect{\beta}$ and to satisfy the usual regularity conditions used in standard asymptotic arguments \citep[cf.][Ch.5]{van2000asymptotic} and which imply the following:
\begin{enumerate} [label=(A\arabic*)]
\item \label{A1} Let $\widehat{\vect{\beta}}_n$ be a consistent root of the system of estimating equations
 \begin{equation}
 \label{ref:estimation}
 v_{\widehat{\vect{\beta}}_n,n}( \vect{\psi}_{\widehat{\vect{\beta}}_n})=0,
 \end{equation}
where $\vect{\psi}_{\vect{\beta}}$ is a $p$-dimensional vector function with linearly independent components in $\mathcal{L}\left(G_{\vect{\beta}}\right)$. 
Denote with $\nabla_{\vect{\beta}}h_{\vect{\beta}}$ the gradient of  $h_{\vect{\beta}}$ taken with respect to $\vect{\beta}$. The Taylor expansion
\begin{equation}\label{eqn:Taylor}
v_{\widehat{\vect{\beta}}_n,n}(h_{\widehat{\vect{\beta}}_n}) = v_{\vect{\beta}_0,n}(h_{\vect{\beta}_0}) +  (\widehat{\vect{\beta}}_n- \vect{\beta}_0)^T \nabla_{\vect{\beta}} v_{\vect{\beta},n}(h_{\vect{\beta}})\Bigl|_{\vect{\beta}=\vect{\beta}_0}
+o_P(1)
\end{equation}
is valid.
\item \label{A2}  For any $h_{\vect{\beta}} \in \mathcal{L}\left(G_{\vect{\beta}}\right)$ and $\vect{\beta} \in \mathcal{N}$, 
$$\nabla_{\vect{\beta}} \int_{\mathcal{X}}  h_{\vect{\beta}}(x)dG_{\vect{\beta}}(x)=\int_{\mathcal{X}} \nabla_{\vect{\beta}} \Big[ h_{\vect{\beta}}(x)dG_{\vect{\beta}}(x)\Big].$$
\end{enumerate}
To ease the intuition, in what follows, we assume \ref{A1}-\ref{A2} directly.  

As first proven by \cite{K2}, when replacing $\vect{\beta}$ with $\widehat{\vect{\beta}}_n$ in $v_{{\vect{\beta}},n}(h_{{\vect{\beta}}})$, the resulting process is asymptotically equal to a projection of the latter evaluated at the true parameter value $\vect{\beta}_0$:
\newtheorem{prop}{Proposition}
\begin{prop}[Khmaladze projection, \citet{K2}]
\label{prop:process_equiv}
If \ref{A1}-\ref{A2} hold, then, under $H_0,$
\begin{equation}
v_{\widehat{\vect{\beta}}_n,n}(h_{\widehat{\vect{\beta}}_n}) = v_{\vect{\beta}_0,n}( \Pi h_{\vect{\beta}_0}) + o_P(1), 
\label{eqn:process_equiv}
\end{equation}
where 
\begin{equation}
\label{eqn:pi}
    \Pi h_{\vect{\beta}} = h_{\vect{\beta}} -\vect{\psi}_{\vect{\beta}}^T\inner{\vect{u}_{\vect{\beta}},\vect{\psi}^T_{\vect{\beta}}}^{-1}_{G_{\vect{\beta}}}\inner{\vect{u}_{\vect{\beta}},h_{\vect{\beta}}}_{G_{\vect{\beta}}}
\end{equation}
is a projection of $h_{\vect{\beta}}$.
\end{prop}
\begin{proof}
We give a proof adapted to the notation and setting of the present manuscript.
Under Assumption \ref{A2}, we have,
\begin{equation}
\begin{aligned}
\frac{1}{\sqrt{n}}
\nabla_{\vect{\beta}} v_{\vect{\beta},n}(h_{\vect{\beta}})
\Big|_{\vect{\beta}=\vect{\beta}_0}
&=
\frac{1}{\sqrt{n}}
\int_{\mathcal{X}}
\nabla_{\vect{\beta}}
\!\left[
h_{\vect{\beta}}(x)\, dv_{\vect{\beta},n}(x)
\right]
\Big|_{\vect{\beta}=\vect{\beta}_0}
\\[0.4em]
&=
\frac{1}{\sqrt{n}}
\int_{\mathcal{X}}
\left[
\nabla_{\vect{\beta}} h_{\vect{\beta}}(x)
\right]\Big|_{\vect{\beta}=\vect{\beta}_0}
\, dv_{\vect{\beta}_0,n}(x)
-
\int_{\mathcal{X}}
h_{\vect{\beta}_0}(x)\,
\left[
\nabla_{\vect{\beta}} g_{\vect{\beta}}(x)
\right]\Big|_{\vect{\beta}=\vect{\beta}_0}
\, dx
\\[0.4em]
&=
-\,
\inner{
\vect{u}_{\vect{\beta}_0},
h_{\vect{\beta}_0}
}_{G_{\vect{\beta}_0}}
+ o_P(1),
\end{aligned}
\label{eqn:gred}
\end{equation}
in which the last equality holds by the law of large numbers. 
Substituting the right-hand side of such an equality into \eqref{eqn:Taylor} gives:
$$v_{\widehat{\vect{\beta}}_n,n}(h_{\widehat{\vect{\beta}}_n}) = v_{\vect{\beta}_0,n}(h_{\vect{\beta}_0}) - \sqrt{n}(\widehat{\vect{\beta}}_n- \vect{\beta}_0)^T \inner{\vect{u}_{\vect{\beta}_0},h_{\vect{\beta}_0}}_{G_{\vect{\beta}_0}}
+o_P(1).$$
Since  the components of $\vect{\psi}_{{\vect{\beta}}}$ belong to $\mathcal{L}\left(G_{\vect{\beta}}\right)$, the asymptotic equality above also holds for $v_{\widehat{\vect{\beta}}_n,n}( \vect{\psi}_{\widehat{\vect{\beta}}_n})$ and equating it to zero leads to:
$$\sqrt{n}(\widehat{\vect{\beta}}_n- \vect{\beta}_0) =  \inner{\vect{u}_{\vect{\beta}_0},\vect{\psi}^T_{\vect{\beta}_0}}^{-1}_{G_{\vect{\beta}_0}}v_{\vect{\beta}_0,n}(\vect{\psi}_{\vect{\beta}_0})
+o_P(1).$$
Combining the last two expressions yields:
$$v_{\widehat{\vect{\beta}}_n,n}(h_{\widehat{\vect{\beta}}_n}) = v_{\vect{\beta}_0,n}(h_{\vect{\beta}_0}) -  v_{\vect{\beta}_0,n}(\vect{\psi}^T_{\vect{\beta}_0})\inner{\vect{u}_{\vect{\beta}_0},\vect{\psi}^T_{\vect{\beta}_0}}^{-1}_{G_{\vect{\beta}_0}}\inner{\vect{u}_{\vect{\beta}_0},h_{\vect{\beta}_0}}_{G_{\vect{\beta}_0}}
+o_P(1),$$
which can be equivalently expressed as in \eqref{eqn:process_equiv}-\eqref{eqn:pi} by linearity of  $\Pi$.
Additionally, it is easy to verify that
$\Pi \Pi h_{\vect{\beta}} = \Pi h_{\vect{\beta}},$
thus, $\Pi$ is a projection operator.
\end{proof}
In general, the projection given by $\Pi$ is not orthogonal; that is because $$\Pi \vect{\psi}_{\vect{\beta}}=0 \quad \text{and}\quad\inner{\vect{u}_{\vect{\beta}},\Pi h_{\vect{\beta}}}_{G_{\vect{\beta}}}=0$$ 
while neither $\Pi \vect{u}_{\vect{\beta}}$ nor $\inner{\vect{\psi}_{\vect{\beta}},\Pi h_{\vect{\beta}}}_{G_{\vect{\beta}}}$ are zero.
From the above expressions, however, it is clear that $\Pi$ defines an orthogonal projection when $\vect{\psi}_{\vect{\beta}}\equiv \vect{u}_{\vect{\beta}}$, that is, when the estimation is conducted via maximum likelihood. 

In the context of smooth tests, Proposition~\ref{prop:process_equiv} implies that the processes $v_{\widehat{\vect{\beta}}_n,n}$ and $v_{{\vect{\beta}}_0,n}$, indexed respectively by $\widetilde{h}_{j\widehat{\vect{\beta}}_n}$ and $\Pi \widetilde{h}_{j\vect{\beta}_0}$, are asymptotically equivalent. Furthermore, the functions $\widetilde{h}_{j\vect{\beta}}$ are orthogonal to $\vect{u}_{\vect{\beta}}$ by construction (see Equation \eqref{eqn:tildeh}). It follows that 
$$\Pi \widetilde{h}_{j\vect{\beta}_0} 
= \widetilde{h}_{j \vect{\beta}_0} - \vect{\psi}_{\vect{\beta}_0}^T\inner{\vect{u}_{\vect{\beta}_0},\vect{\psi}^T_{\vect{\beta}_0}}^{-1}_{G_{\vect{\beta}_0}}\bigl\langle{\vect{u}_{\vect{\beta}_0},\widetilde{h}_{j \vect{\beta}_0}\bigl\rangle}_{G_{\vect{\beta}_0}} 
= \widetilde{h}_{j\vect{\beta}_0}.$$ 
This fact, combined with the asymptotic equality in \eqref{eqn:process_equiv}, brings us to the following proposition.
\begin{prop} \label{prop3}
Assume that \ref{A1} and \ref{A2} hold for each $\widetilde{h}_{j\vect{\beta}}$, $j\in\mathcal{M}_n$, then 
\begin{equation}\label{prop:3_result}
v_{\vect{\beta}_0,n}( \widetilde{h}_{j\vect{\beta}_0}) = v_{\vect{\beta}_0,n}(\Pi \widetilde{h}_{j\vect{\beta}_0}) = v_{\widehat{\vect{\beta}}_n,n}(\widetilde{h}_{j\widehat{\vect{\beta}}_n}) +o_P(1). 
\end{equation}
\end{prop}
To understand the computational advantages entailed by Proposition~\ref{prop3}, denote with $\widehat{\vect{\beta}}_{\text{obs}}$ the estimate of $\vect{\beta}$ obtained on the observed data by solving \eqref{ref:estimation}. In the parametric bootstrap, such an estimate plays the same role as $\vect{\beta}_0$ in \eqref{prop:3_result}. Let $\widehat{\vect{\beta}}_n^{(b)}$ be the parameter estimate obtained on the $b$-th bootstrap sample generated from  $G_{\widehat{\vect{\beta}}_{\text{obs}}}.$ Equation~\eqref{prop:3_result}  implies that the empirical process indexed by functions    $\widetilde{h}_{j\widehat{\vect{\beta}}^{(b)}_n}$ is asymptotically equal to the process indexed by $\widetilde{h}_{j\widehat{\vect{\beta}}_{\text{obs}}}$. In other words, the projection structure characterizing the functions $\widetilde{h}_{j\widehat{\vect{\beta}}_{\text{obs}}}$ makes the effect of re-estimating the parameter asymptotically negligible. Thus, for sufficiently large $n$, the user must evaluate the process $v_{\widehat{\vect{\beta}}_{\text{obs}},n}(\widetilde{h}_{j\widehat{\vect{\beta}}_{\text{obs}}})$ over different bootstrap samples but does not need to repeat the estimation of $\vect{\beta}$ at each replicate. 
It follows that this \emph{projected bootstrap} procedure 
can always make the simulation more efficient. As demonstrated in the next section with an example, the projected bootstrap is particularly advantageous when parameter estimation is time-consuming. 

Finally, the consistency of the projected bootstrap in recovering the true distribution of the function-parametric empirical process $v_{\vect{\beta}_0,n}(\widetilde{h}_{j\vect{\beta}_0})$ is guaranteed under the same conditions needed for the classical parametric bootstrap \citep[cfr.][]{babu2004}.

\subsection{Empirical illustrations of smooth tests via the projected bootstrap} \label{sec:3.1}
Let $G_{\vect{\beta}}$  be the asymmetric Laplace distribution with unknown asymmetry parameter $\vect{\beta}$. The corresponding PDF, $g_{\vect{\beta}}$, can be expressed as:
\begin{equation}
    \begin{aligned}
g_{\vect{\beta}}(x)=\begin{cases}\frac{\sqrt{2}}{\sigma} \frac{\vect{\beta}}{1+\vect{\beta}^2} \exp \left(-\frac{\sqrt{2} \vect{\beta}}{\sigma}(x-\theta)\right), & \text { for }  x \geq \theta \\ \frac{\sqrt{2}}{\sigma} \frac{\vect{\beta}}{1+\vect{\beta}^2} \exp \left(\frac{\sqrt{2}}{\sigma \vect{\beta}}(x-\theta)\right), & \text { for }  x<\theta\end{cases}.
\label{eqn:den}
    \end{aligned}
\end{equation}
In what follows, the parameters $\theta$ and $\sigma$ are set equal to 0 and 2, respectively, and are assumed to be known.

\begin{figure}%
    \centering
    \includegraphics[width=0.49\textwidth]{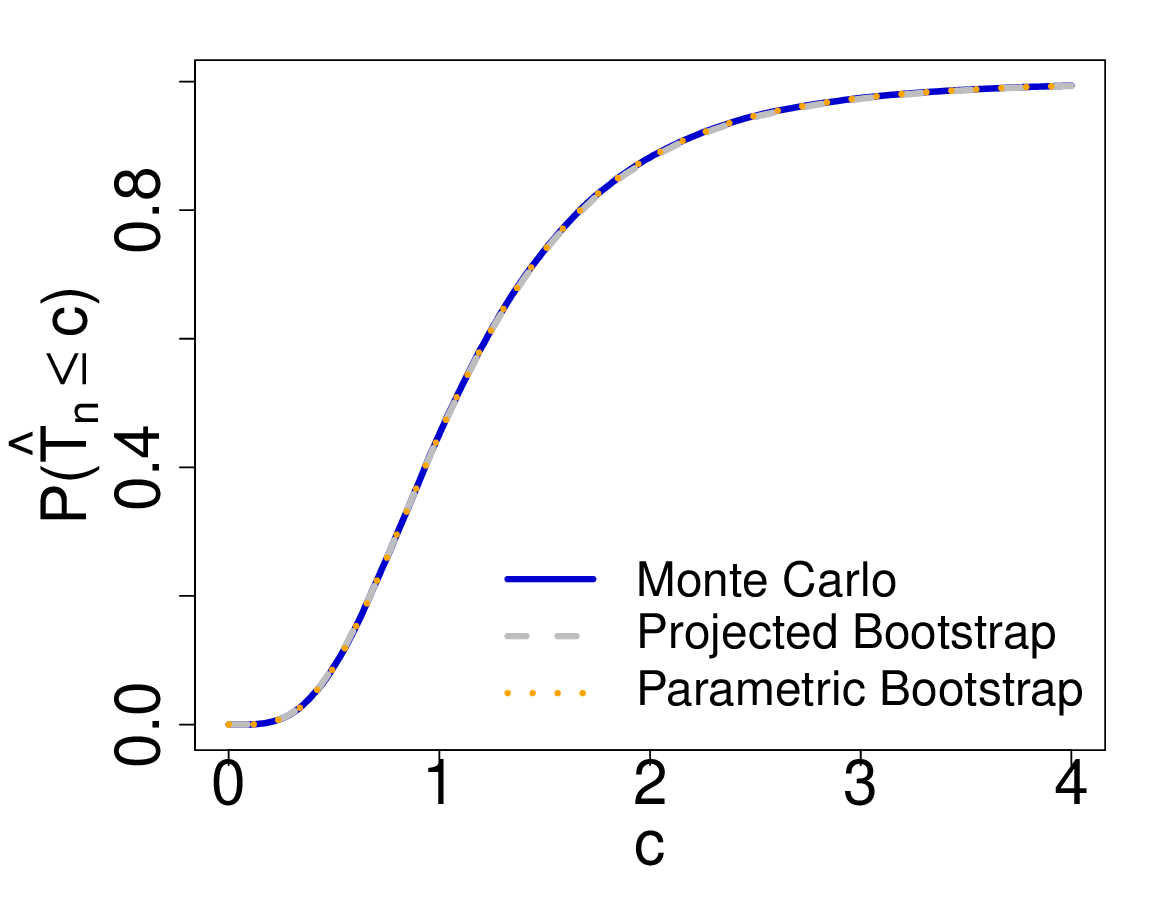} 
    \includegraphics[width=0.49\textwidth]{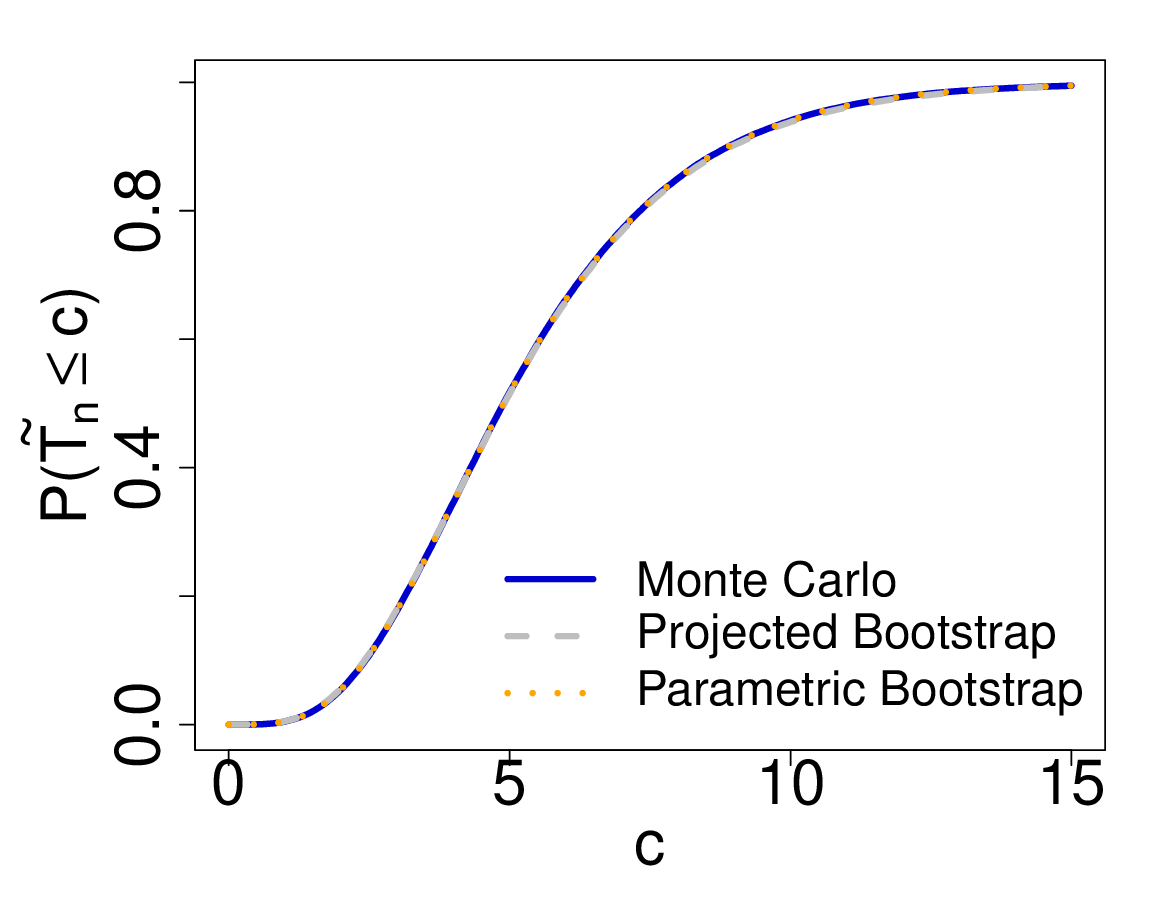}
    \caption{Comparing the simulated null distributions of the order selection test statistic in \eqref{eqn:ordersel} (Left), and subset selection test statistic in \eqref{eqn:subsetsel} (Right) via the parametric bootstrap (orange dotted lines), the projected bootstrap (grey dashed lines), and   Monte Carlo  (blue solid lines). }
 \label{fig:image2}
\end{figure}
A dataset $\bm{x}_n=(x_1,\dots,x_n)$ of $n=100$ observations is generated from the density specified in \eqref{eqn:den}, with the true value of the asymmetry parameter set to be $\vect{\beta}_0=0.1$. In the present simulation study, such a dataset plays the same role as the observed data sample in a real-data analysis.  The maximum likelihood estimator $\widehat{\vect{\beta}}_{n}$ of $\vect{\beta}$ based on $\bm{x}_n$ is obtained by solving the following equation \citep{kotz2002}: 
\begin{equation}
    \begin{aligned}
& 1-\frac{2 \vect{\beta}^2}{\left(1+\vect{\beta}^2\right)} +\frac{\sqrt{2}}{\sigma}\left[\frac{1}{n\vect{\beta}} \sum_{j=1}^{n}\left(x_j-\theta\right)^{-} - \frac{\vect{\beta}}{n} \sum_{j=1}^{n}\left(x_j-\theta\right)^{+} \right]=0,
\label{eq:solve}
\end{aligned}
\end{equation}
where 
\begin{equation*}
\left(x_j-\theta\right)^{-}=-\min (0, x_j-\theta) \quad \text{and}  \quad (x_j-\theta)^{+}=\max (0, x_j-\theta).
\end{equation*}
Given the complexity of finding the solution of \eqref{eq:solve}, it is anticipated that using the projected bootstrap -- in which the estimator is obtained only once rather than recalculated on each bootstrap sample -- will expedite the simulation procedure. 

To test the hypothesis in \eqref{eqn:h0h1}, we consider the order selection statistic in \eqref{eqn:ordersel} and the subset selection statistic in \eqref{eqn:subsetsel}. 
Their null distributions are simulated via three methods: the classical parametric bootstrap, the projected bootstrap, and Monte Carlo with data generated from the true distribution. By incorporating a simulation of Monte Carlo samples, we aim to determine whether the two bootstrap procedures considered can accurately recover the true null distribution of the test statistics for a sample size of $ n = 100$. All simulations are conducted using $100,000$ replicates. The orthonormal basis functions $\{h_{j\vect{\beta}}\}_{j=1}^{M_n}$ are chosen as  compositions of the normalized shifted Legendre polynomials $h_j$ on $[0,1]$ with the CDF $G_{\vect{\beta}}$, i.e., $h_{j\vect{\beta}} = h_j\circ G_{\vect{\beta}}$. The maximum number of basis functions  $M_n$ considered is 10. 

The results of the simulation, presented in Figure \ref{fig:image2},
demonstrate that the simulated null distributions of the order selection and subset selection test statistics are consistent across all three methods considered. This indicates that both bootstrap procedures successfully recover the true null distribution of the test statistics, even with a sample size as small as 100 observations. In terms of computational efficiency, however, the projected bootstrap method is significantly faster than the classical parametric bootstrap. Specifically, 
simulating the null distributions of the order selection and subset selection statistics using 100,000 replicates required approximately 9 minutes of CPU time using the projected bootstrap, less than a fourth of the time needed by the classical parametric bootstrap, which required approximately 37 minutes.

The p-values for testing \eqref{eqn:h0h1} are also computed by comparing the value of the order selection and subset selection statistics calculated on $\bm{x}_n$
with their null distribution simulated by means of the projected bootstrap. For the order selection statistic, the p-value obtained is 0.843, and for the subset selection statistic, it amounts to 0.963. 

\section{Distribution-free smooth tests via K2 transform} \label{sec:4}
In general, statistics given by functionals of the process $v_{\vect{\beta},n}(\widetilde{h}_{j\vect{\beta}})$ are not distribution-free. 
Thus, while the projected bootstrap introduced in Section~\ref{sec:3} can reduce the computational burden of simulating their null distribution, a different simulation must be implemented for each model being tested. 

This section demonstrates that another route to distribution-freeness is available to practitioners. Instead of combining the values of the process $v_{\vect{\beta},n}(\widetilde{h}_{j\vect{\beta}})$ in a rather specific -- and possibly forceful -- manner to construct, at most, a few distribution-free statistics, the distribution-free property can be retrieved by ensuring that the empirical process $v_{\vect{\beta},n}(\widetilde{h}_{j\vect{\beta}})$ is itself asymptotically distribution-free under the null. Such an approach guarantees that the limiting null distribution of all its functionals is also distribution-free, thereby providing the user with an entire family of asymptotically distribution-free statistics. The tool that enables such a construction is the so-called Khmaladze-2 (K2) transform introduced by \cite{K22016}. 

In classical nonparametric goodness-of-fit testing, the probability integral transform is the transformation commonly employed to map the empirical process into the uniform empirical process, which is known to be asymptotically distributed as a standard Brownian bridge. Distribution-freeness is therefore achieved by means of a simple change of variable. One can think of the K2 transform also as a change of variable, but in functional space.  In the parametric setting, such a change of variable allows us to retrieve asymptotic distribution-freeness by mapping a wide range of different projected Brownian motions, arising when estimating different distributions, into the same ``standard'' projected Brownian motion of choice. 
Moreover, when applied in the context of smooth tests, the K2 transform implicitly guides the construction of an orthonormal basis for  $L^2(G_{\vect{\beta}})$ that guarantees the distribution-freeness of the empirical process indexed by functions in such a basis, as well as that of all its functionals.

Let $F_{\vect{\gamma}}$ be a ``reference distribution''  
that has the same support $\mathcal{X}$ as $G_{\vect{\beta}}$. 
For the moment, assume that $\vect{\gamma}$ and $\vect{\beta}$ have the same dimension; an extension to the case where $\vect{\gamma}$ has fewer parameters than $\vect{\beta}$ will be discussed in Section \ref{sec:4.1}. The distribution $F_{\vect{\gamma}}$ constitutes the starting point of our construction and should be chosen to be simple -- that is, easy to simulate from, differentiate, and evaluate. 
Consider the Hilbert space $L^2(F_{\vect{\gamma}}) = \{ \phi: \inner{\phi,\phi}_{F_{\vect{\gamma}}}<\infty\}$,
with inner product
$$\inner{\phi,\phi'}_{F_{\vect{\gamma}}} = \int_{\mathcal{X}}\phi(x)\phi'(x)dF_{\vect{\gamma}}(x),$$
and the empirical process 
\begin{equation*}
    \begin{aligned}
    &v_{\gamma,n}(\phi) = \int_{\mathcal{X}}\phi(x)dv_{\vect{\gamma},n}(x) = 
    \frac{1}{\sqrt{n}}\sum_{i=1}^n \bigl[\phi (X_i)- \mathbb{E}_{F_{\vect{{\gamma}}}}[\phi (X_i)]
\bigl]
    \label{eqn:vF}
    \end{aligned}
    \end{equation*}
indexed by such functions. Let $\vect{a}_{\vect{\gamma}}=[a_{\gamma_1},\cdots, a_{{\gamma_p}}]^T$ be the orthonormalized score function of $F_{\vect{\gamma}}$. Denote with  
$\mathcal{L}\left(F_{\vect{\gamma}}\right)$ and $\mathcal{L}_{\perp}\left(F_{\vect{\gamma}}\right)$ the subspaces of $L^2\left(F_{\vect{\gamma}}\right)$ given by
$$\mathcal{L}\left(F_{\vect{\gamma}}\right)=\left\{ \phi \in L^2\left(F_{\vect{\gamma}}\right): \inner{\phi,\bm{1}}_{F_{\vect{\gamma}}}=0\right\}, \ \ \
\mathcal{L}_{\perp}\left(F_{\vect{\gamma}}\right)=\left\{ \widetilde{\phi} \in \mathcal{L}\left(F_{\vect{\gamma}}\right): \bigl\langle{\widetilde{\phi},\vect{a}_{\vect{\gamma}}\bigl\rangle}_{F_{\vect{\gamma}}}=0\right\}.$$
For any given orthonormal basis $\{\phi_{j\vect{\gamma}}\}_{j=1}^{\infty}$ of $L^2(F_{\vect{\gamma}})$, with elements in $\mathcal{L}(F_{\vect{\gamma}})$, the empirical process $v_{\gamma,n}$ indexed by the residuals
$$\widetilde{\phi}_{j\vect{\gamma}} = \phi_{j\vect{\gamma}} -\vect{a}^T_{\vect{\gamma}}\left\langle \vect{a}_{\vect{\gamma}},\phi_{j\vect{\gamma}}\right\rangle_{F_{\vect{\gamma}}}, $$
is asymptotically distributed, under $F_{\vect{\gamma}},$ as a projected Brownian motion with mean and covariance given by
\begin{align*}
\bigl\langle{\widetilde{\phi}_{j,\vect{\gamma}},\vect{1}\bigl\rangle}_{F_{\vect{\gamma}}}&=0\quad\text{and}\\
\bigl\langle{\widetilde{\phi}_{i\vect{\gamma}},\widetilde{\phi}_{j\vect{\gamma}}\bigl\rangle}_{F_{\vect{\gamma}}} &= \mathds{1}_{\{i=j\}}-\sum_{k=1}^p \inner{a_{\vect{\gamma_k}},\phi_{i\vect{\gamma}}}_{F_{\vect{\gamma}}}\bigl\langle{a_{\vect{\gamma_k}},\phi_{j\vect{\gamma}}\bigl\rangle}_{F_{\vect{\gamma}}},
\end{align*}
respectively.
Such a limiting distribution is the ``standard'' distribution that will be recovered through the K2 transform.
In particular, since a Gaussian process is fully characterized by its mean and covariance functions, the K2 transform enables the construction of an orthonormal basis $\{h_{j\vect{\beta}}\}_{j=1}^\infty$ for $L^2(G_{\vect{\beta}})$, hereinafter referred to as \emph{K2 orthonormal basis}, whose  residuals, $\{\widetilde{h}_{j\vect{\beta}}\}_{j=1}^\infty$ satisfy
\begin{equation*}
\begin{aligned}
\bigl\langle{\widetilde{h}_{j\vect{\beta}},\vect{1}\bigl\rangle}_{{G_{\vect{\beta}}}}= \bigl\langle{\widetilde{\phi}_{j\vect{\gamma}},\vect{1}\bigl\rangle}_{{F_{\vect{\gamma}}}} = 0 \quad  \text { and } \quad
\bigl\langle{\widetilde{h}_{j\vect{\beta}}, \widetilde{h}_{s\vect{\beta}}\bigl\rangle}_{G_{\vect{\beta}}} = \bigl\langle{\widetilde{\phi}_{j\vect{\gamma}}, \widetilde{\phi}_{s\vect{\gamma}}\bigl\rangle}_{F_{\vect{\gamma}}}.
\label{eqn:meanvar}
\end{aligned}
\end{equation*}
It follows that the processes $v_{\vect{\beta},n}(\widetilde{h}_{j\vect{\beta}})$ and $v_{\vect{\gamma},n}(\widetilde{\phi}_{j\vect{\gamma}})$, as well as all their functionals, have the same limiting distribution under $G_{\vect{\beta}}$ and $F_{\vect{\gamma}}$, respectively, thereby enabling the construction of a large class of asymptotically distribution-free statistics for testing \eqref{eqn:h0h1}.

For example, let $\widehat{\vect{\beta}}_n$ and $\widehat{\vect{\gamma}}_n$ denote the estimators of $\vect{\beta}$ and $\vect{\gamma}$, respectively.
Consider the order-selection and subset-selection statistics defined in \eqref{eqn:ordersel} and \eqref{eqn:subsetsel}, in which $\widehat{S}_{m,n}$ is chosen to be the (unnormalized) generalized score statistic in \eqref{eqn:non_norm}, with $\thjn$ corresponding to the residuals of the K2 orthonormal basis. Under $G_{\vect{\beta}}$, such statistics have the same limiting distribution as
the statistics
\begin{equation}
\begin{aligned}
    \max_{{m\in \mathcal{M}_n}} \sum_{j=1}^m\frac{ v^2_{\widehat{\vect{\gamma}}_n,n}(\widetilde{\phi}_{j\widehat{\vect{\gamma}}_n})}{m} \ \ \text{ and } \ \ 
    \max_{{B \subseteq \mathcal{M}_n:B\neq \emptyset}}\sum_{j\in B} \frac{v^2_{\widehat{\vect{\gamma}}_n,n}(\widetilde{\phi}_{j\widehat{\vect{\gamma}}_n})}{|B|},
\label{eqn:K2st2}
\end{aligned}
\end{equation}
under $F_{\vect{\gamma}}$. Since the latter can be chosen arbitrarily, the limiting null distribution of the statistics in \eqref{eqn:K2st2} can be easily simulated by means of the projected bootstrap described in Section~\ref{sec:3}. This also implies that, when testing different models, the standard limiting null distribution of the corresponding processes can be obtained using a single simulation conducted under $F_{\vect{\gamma}}$. 

The construction of the K2 orthonormal basis and validation of its properties are described in detail in Section \ref{sec:4.1}; whereas, its effectiveness in retrieving distribution-freeness in finite samples is investigated in Section \ref{sec:4.2} through a suite of simulation studies. 

\subsection{On the construction of the K2 orthonormal basis}\label{sec:4.1}

Let $l_{\vect{\gamma},\vect{\beta}}$ be the function defined by
\begin{equation*}
\begin{aligned}
l_{\vect{\gamma},\vect{\beta}}(x)= \sqrt{\frac{ f_{\vect{\gamma}}(x)}{g_{\vect{\beta}}(x)}}, \quad x\in\mathcal{X}.
\label{eqn:isometry}
\end{aligned}
\end{equation*}
The multiplication operator $\phi\mapsto l_{\vect{\gamma},\vect{\beta}}\phi$ defines an isometry from
$L^2(F_{\vect{\gamma}})$ to $L^2(G_{\vect{\beta}})$. 

It can be easily verified that for all $j$'s,
\begin{equation*} 
\inner{l_{\vect{\gamma},\vect{\beta}} \phj,\vect{1}}_{{G_{\vect{\beta}}}}\neq 0, \quad \text{and} \quad \inner{l_{\vect{\gamma},\vect{\beta}}\phj,l_{\vect{\gamma},\vect{\beta}}\phs}_{{G_{\vect{\beta}}}}=\inner{\phj,\phs}_{{F_{\vect{\gamma}}}}.
\end{equation*}
It follows that the functions $l_{\vect{\gamma},\vect{\beta}}\phj$ have the same covariance as $\phj$ under $F_{\vect{\gamma}}$, but do not have the same mean. Thus, they belong to $L^2(G_{\vect{\beta}})$ but not to its subspace $\mathcal{L}(G_{\vect{\beta}})$.  

To rectify this, consider a linear operator, $K$, such that, when applied to $l_{\vect{\gamma},\vect{\beta}}\phj$, the resulting functions have mean zero under $G_{\vect{\beta}}$ but, at the same time, their covariance is preserved, that is, 
\begin{equation}
\label{eqn:zeromean}
\inner{K l_{\vect{\gamma},\vect{\beta}}\phj,\vect{1}}_{{G_{\vect{\beta}}}}= 0, \quad 
\inner{K l_{\vect{\gamma},\vect{\beta}}\phj,K l_{\vect{\gamma},\vect{\beta}}\phs}_{{G_{\vect{\beta}}}}=\inner{\phj,\phs}_{{F_{\vect{\gamma}}}}. 
\end{equation}
The second condition in \eqref{eqn:zeromean} is satisfied for any  $K$ which is unitary. For what concerns the first condition, note that 
$$\inner{l_{\vect{\gamma},\vect{\beta}}\phj,l_{\vect{\gamma},\vect{\beta}}}_{{G_{\vect{\beta}}}}=\inner{\phj,\vect{1}}_{{F_{\vect{\gamma}}}} = 0;$$
hence, if  $Kl_{\vect{\gamma},\vect{\beta}}=\vect{1}$, then
$$\inner{K l_{\vect{\gamma},\vect{\beta}}\phj,\vect{1}}_{{G_{\vect{\beta}}}}= \inner{K l_{\vect{\gamma},\vect{\beta}}\phj,Kl_{\vect{\gamma},\vect{\beta}}}_{{G_{\vect{\beta}}}}= \inner{ l_{\vect{\gamma},\vect{\beta}}\phj,l_{\vect{\gamma},\vect{\beta}}}_{{G_{\vect{\beta}}}}=0.$$
A unitary operator $K$ satisfying the above requirements is the reflection
\begin{equation}
\begin{aligned}
K= I- 2\frac{l_{\vect{\gamma},\vect{\beta}}-\vect{1}}{||l_{\vect{\gamma},\vect{\beta}}-\vect{1}||^2}\langle l_{\vect{\beta} ,\vect{\gamma}}-\vect{1}, \ \cdot \ \rangle_{G_{\vect{\beta}}} 
&= I -\frac{l_{\vect{\gamma},\vect{\beta}}-\vect{1}}{1-\langle l_{\vect{\gamma},\vect{\beta}}, \vect{1}\rangle_{G_{\vect{\beta}}}}\langle l_{\vect{\beta} ,\vect{\gamma}}-\vect{1}, \ \cdot \ \rangle_{G_{\vect{\beta}}},
\label{eqn:Kop}
\end{aligned}
\end{equation}
where $I$ is an identity operator.
The properties of $K$ are summarized in Proposition~\ref{prop:K} and verified in Appendix~\ref{App:unitary}. 
\begin{prop}[\citealt{K22016}] 
\label{prop:K}
The operator $K$ is unitary, and satisfies $$ K l_{\vect{\gamma},\vect{\beta}}=\vect{1} \quad\text{and}\quad K \vect{1}=l_{\vect{\gamma},\vect{\beta}}.$$
Moreover, for any  $\zeta\in L^2(G_{\vect{\beta}})$ such that $\inner{\zeta, l_{\vect{\gamma},\vect{\beta}}}_{G_{\vect{\beta}}}=\inner{\zeta, \vect{1}}_{G_{\vect{\beta}}}=0$,  $K \zeta=\zeta.$  
\end{prop} 

Thus far, we have demonstrated that the functions $K l_{\vect{\gamma},\vect{\beta}} \phj\in \mathcal{L}(G_{\vect{\beta}})$ have the same mean and covariance as the functions $\phj\in \mathcal{L}(F_{\vect{\gamma}})$. Therefore, when  $\vect{\beta}$ is known, test statistics constructed as functionals of the empirical process $v_{\vect{\beta},n}(h_{j\vect{\beta}})$  are asymptotically distribution-free if  $h_{j\vect{\beta}} = K l_{\vect{\gamma},\vect{\beta}} \phj$. 

When $\vect{\gamma}$ and $\vect{\beta}$ are unknown, an additional step is necessary to map functions $\tphj\in \mathcal{L}_\perp(F_{\vect \gamma})$ into functions $\thj\in \mathcal{L}_\perp(G_{\vect \beta})$ with the same mean and covariance. For instance, if we were to naively choose $\thj$ to be 
$$\thj= K l_{\vect{\gamma},\vect{\beta}} \phj -\vect{b}^T_{\vect{\beta}}\inner{\vect{b}_{\vect{\beta}},K l_{\vect{\gamma},\vect{\beta}} \phj}_{G_{\vect{\beta}}}, $$ 
then, in general,
\begin{equation*}
\begin{aligned}
\langle \thj, \ths \rangle_{G_{\vect{\beta}}}= \inner{ \phj, \phs}_{F_{\vect{\gamma}}} - \inner{  K l_{\vect{\gamma},\vect{\beta}}\phj, \vect{b}^T_{\vect{\beta}}}_{G_{\vect{\beta}}}\inner{ \vect{b}_{\vect{\beta}}, K l_{\vect{\gamma},\vect{\beta}} \phs}_{G_{\vect{\beta}}} 
\neq
\langle\tphj,\tphs\rangle_{F_{\vect{\gamma}}}. \end{aligned}
\end{equation*}
In the above expression, equality holds, however, if 
$$\big\langle{ \vect{b}_{\vect{\beta}}, K l_{\vect{\gamma},\vect{\beta}}\phj\big\rangle}_{G_{\vect{\beta}}}  = \inner{\vect{a}_{\vect{\gamma}},\phj}_{F_{\vect{\gamma}}}$$
for all $j$'s.
This motivates the construction of an operator, $\vect{U}_p$, such that, when applied to the functions $K l_{\vect{\gamma},\vect{\beta}} \phj$, guarantees
\begin{equation} \label{eqn:au_bh}
\inner{ \vect{b}_{\vect{\beta}}, \vect{U}_p K l_{\vect{\gamma},\vect{\beta}}\phj}_{G_{\vect{\beta}}} = \inner{\vect{a}_{\vect{\gamma}},\phj}_{F_{\vect{\gamma}}}, 
\end{equation}
and at the same time preserves the mean and covariance structure, i.e., 
\begin{equation} \label{eqn:U_meanvar}
\inner{\vect{U}_p K l_{\vect{\gamma},\vect{\beta}}\phj, \vect{1}}_{G_{\vect{\beta}}} = 0, \quad \inner{\vect{U}_p K l_{\vect{\gamma},\vect{\beta}} \phj, \vect{U}_p K l_{\vect{\gamma},\vect{\beta}}\phs}_{G_{\vect{\beta}}} = \inner{\phj,\phs}_{F_{\vect{\gamma}}}.
\end{equation}
Let $c_{\vect{\lambda}_k}$ be the $k$th component of the vector  function $K l_{\vect{\gamma},\vect{\beta}} \vect{a}_{\vect{\gamma}}$, i.e., 
$$c_{\vect{\lambda}_k} = K l_{\vect{\gamma},\vect{\beta}} a_{\vect{\gamma}_k}, \quad k=1,\dots,p.$$
While any unitary operator $\vect{U}_p$ such that $\vect{U}_p\vect{1}=\vect{1}$ fulfills the requirements in \eqref{eqn:U_meanvar}, to ensure \eqref{eqn:au_bh} holds,  $\vect{U}_p$ must be choosen so that  $\vect{U}_pc_{\vect{\lambda}_k}=b_{\vect{\beta}_k}$, for each $k=1,\dots,p$. 

One could naively attempt to construct $\vect{U}_p$ as a composition of $p$ unitary operators, each mapping $c_{\vect{\lambda}_k}$ to $b_{\vect{\beta}_k}$; such an approach, however, would not lead to the desired result. To see that, consider the reflection operator on $ L^2(G_{\vect{\beta}})$ defined as 
\begin{equation*}
\begin{aligned}
U_{b_{\vect{\beta}_k} c_{\lambda_k}} &=
I- 2\frac{b_{\vect{\beta}_k}-c_{\lambda_k}}{||b_{\vect{\beta}_k}-c_{\lambda_k}||^2}\langle b_{\vect{\beta}_k}-c_{\lambda_k}, \ \cdot \ \rangle_{G_{\vect{\beta}}} \\ 
&= I -\frac{
b_{\vect{\beta}_k}-c_{\lambda_k}}{\vect{1}-\left\langle b_{\vect{\beta}_k}, c_{\lambda_k}\right\rangle_{G_{\boldsymbol{\beta}}}}\left\langle b_{\vect{\beta}_k}-c_{\lambda_k}, \ \cdot \ \right\rangle_{G_{\boldsymbol{\beta}}},  \quad k=1,\cdots,p. 
\end{aligned}
\end{equation*} 
Such an operator is self-adjoint and unitary on $L^2(G_{\vect{\beta}})$. It maps $c_{\vect{\lambda}_k}$ to $b_{\vect{\beta}_k}$ and $b_{\vect{\beta}_k}$ to $c_{\vect{\lambda}_k}$, while leaving all functions orthogonal to both $b_{\vect{\beta}_k}$ and $c_{\vect{\lambda}_k}$ unchanged. Let $p=2$, then
$$ U_{b_{\vect{\beta}_2} c_{\lambda_2}}\circ U_{b_{\vect{\beta}_1} c_{\lambda_1}}c_{\lambda_{1}}=U_{b_{\vect{\beta}_2} c_{\lambda_2}} b_{\vect{\beta}_1},$$
and in general 
$$ U_{b_{\vect{\beta}_2} c_{\lambda_2}} b_{\vect{\beta}_1}  \neq b_{\vect{\beta}_1} \quad  \text{unless} \quad\inner{c_{\lambda_2},  b_{\vect{\beta}_1}}_{G_{\vect{\beta}}}=0.$$ 

To address this issue, set $\widetilde{c}_{\lambda_1}=c_{\lambda_1}$ and a set of functions $\{\widetilde{c}_{{\lambda}_k}\}_{k=2}^p$, where each $\widetilde{c}_{{\lambda}_k}$ is constructed to be orthogonal to $\vect{1}$ and to every $b_{\vect{\beta}_j}$ for which $j\leq k-1$. Specifically, we set 
\begin{equation}
    \begin{aligned}
\widetilde{c}_{\lambda_{k}}&=U_{b_{\vect{\beta}_{k-1}}\widetilde{c}_{\lambda_{k-1}}} \circ \cdots  \circ U_{b_{\vect{\beta}_1} \widetilde{c}_{\lambda_1}} c_{\lambda_{k}}.
\label{eqn:tildec}
\end{aligned}
\end{equation}
A proof that each $\widetilde{c}_{\lambda_{k}}$ satisfies the required orthogonality conditions is provided in Appendix~\ref{App:orthor}. 

Now, define the operator $\vect{U}_p$ as 
\begin{equation*}
    \begin{aligned}
\vect{U}_p =U_{b_{\vect{\beta}_p} \widetilde{c}_{\lambda_p}} \circ \cdots  \circ U_{b_{\vect{\beta}_1} \widetilde{c}_{\lambda_1}}.
\label{eqn:Ud}
\end{aligned}
\end{equation*}
Since $\vect{U}_p$ is a composition of unitary operators on $L^2\left(G_{\vect{\beta}}\right)$, it is itself unitary on $L^2\left(G_{\vect{\beta}}\right)$. Moreover, 
\begin{equation*}
    \vect{U}_{p}c_{\lambda_k} = U_{b_{\vect{\beta}_p} \widetilde{c}_{\lambda_p}} \circ \cdots \circ U_{b_{\vect{\beta}_{k}} \widetilde{c}_{\lambda_{k}}} \widetilde{c}_{\lambda_{k}} = U_{b_{\vect{\beta}_p} \widetilde{c}_{\lambda_p}} \circ \cdots \circ U_{b_{\vect{\beta}_{k+1}} \widetilde{c}_{\lambda_{k+1}}} b_{\vect{\beta}_k} = b_{\vect{\beta}_k},
\end{equation*}
thus, $\vect{U}_{p}$ maps $ c_{\lambda_k} $ into $b_{\vect{\beta}_k}$, for each $k=1,\dots,p$ and, since all the functions $b_{\vect{\beta}_k}$ and $\widetilde{c}_{\lambda_k}$
are orthogonal to
$\vect{1}$, $\vect{U}_{p}\vect{1}=\vect{1}$. These properties are formalized in Proposition~\ref{prop:Up}. 

\begin{prop} \label{prop:Up}
The operator $\vect{U}_p$ is unitary on $L^2\left(G_{\vect{\beta}}\right)$ and satisfies
$$
\vect{U}_p\vect{1}=\vect{1},
\ \ \text {and } \ \  \vect{U}_{p}c_{\lambda_k} = b_{\vect{\beta}_k}, \quad k=1,\dots,p. $$
\end{prop} 
From Proposition~\ref{prop:Up}, it follows that each function $\vect{U}_{p}K l_{\vect{\gamma},\vect{\beta}} \phj$ fulfills \eqref{eqn:au_bh}, thus:
\begin{equation*}
\begin{aligned}
\vect{U}_{p}Kl_{\vect{\gamma},\vect{\beta}} \tphj&= \vect{U}_{p} Kl_{\vect{\gamma},\vect{\beta}}\left[\phj-\vect{a}^T_{\vect{\gamma}}\inner{\vect{a}_{\vect{\gamma}}, \phj}_{F_{\vect{\gamma}}}\right]  \\
&=\vect{U}_{p}K l_{\vect{\gamma},\vect{\beta}} \phj-\vect{b}^T_{\vect{\beta}}\inner{\vect{b}_{\vect{\beta}}, \vect{U}_{p}K l_{\vect{\gamma},\vect{\beta}} \phj}_{G_{\vect{\beta}}}. 
\end{aligned}
\end{equation*}

We now have access to all the elements needed to define the K2 orthonormal basis at the core of the proposed distribution-freeness construction:
\begin{prop} \label{prop:6}
The set of functions $\{\vect{U}_{p}K l_{\vect{\gamma},\vect{\beta}} \phj\}_{j=1}^\infty$ forms an orthonormal basis for $L^2(G_{\vect{\beta}})$. 
Define, for each $j \geq 1, $ 
\begin{equation*}
\hj = \vect{U}_{p}K l_{\vect{\gamma},\vect{\beta}} \phj, \quad \text{and} \quad \thj = \vect{U}_{p}K l_{\vect{\gamma},\vect{\beta}} \tphj.
\end{equation*}
Then, the empirical processes $v_{\vect{\beta},n}(\thj)$ and $v_{\vect{\gamma},n}(\tphj)$ have the same limiting distribution under $G_{\vect{\beta}}$ and $F_{\vect{\gamma}}$, respectively.
\end{prop} 
\begin{proof}
The first claim follows from the unitary nature of the operators $\vect{U}_{p}$, $K$, and $l_{\vect{\gamma},\vect{\beta}}$.  
The second claim is a direct consequence of the asymptotic Gaussianity of $v_{\vect{\beta},n}(\thj)$ and $v_{\vect{\gamma},n}(\tphj)$ and the construction of the $\{\thj\}_{j=1}^\infty$ basis, which ensures that the empirical process indexed by such functions has, under $G_{\vect{\beta}}$, the same mean and covariance as the empirical process indexed by functions $\{\tphj\}_{j=1}^\infty$ under $F_{\vect{\gamma}}$.
\end{proof}


The approach described above can be generalized to cases where the dimension of $\vect{\gamma}$ is smaller than $p$, say $s$. In this setting, one can simply expand the orthonormal set of score functions $\{a_{\vect{\gamma}_k}\}_{k=1}^s$ to a larger orthonormal set $\{a_{\vect{\gamma}_k}\}_{k=1}^p$ in $L^2(F_{\vect{\gamma}})$, ensuring that all elements remain orthogonal to each other and the constant function $\vect{1}$. This extension can be accomplished, for instance, by selecting $p-s$ additional functions from another orthonormal basis in $L^2(F_{\vect{\gamma}})$ outside the span of $\{\vect{1},a_{\vect{\gamma}_1},\dots,a_{\vect{\gamma}_s}\}$, and applying the Gram-Schmidt orthogonalization procedure.


\subsection{Simulation Studies} \label{sec:4.2}

\begin{figure}[t]
    \centering
\includegraphics[width=0.7\textwidth]{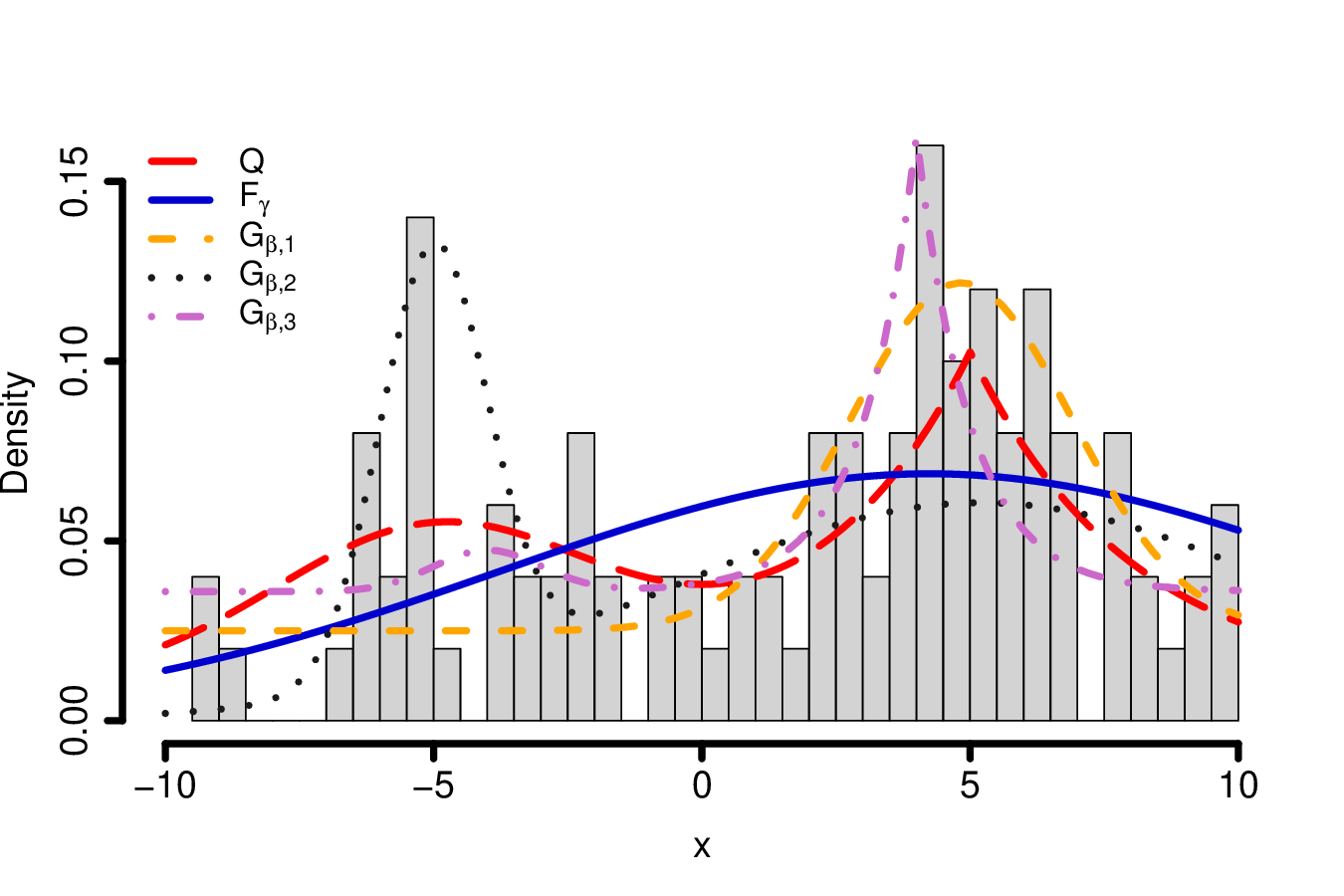}
\caption{The histogram of the simulated dataset is shown together with the densities of the true data-generating model $Q$, the reference distribution $F_{\vect{\gamma}}$, and the hypothesized distributions $G_{\vect{\beta},1}, G_{\vect{\beta},2}, G_{\vect{\beta},3}$. The unknown parameters $\vect{\beta}$ and $\vect{\gamma}$ are estimated via maximum likelihood. }
 \label{fig:histogram1}
\end{figure}

Consider a dataset of $n = 100$ observations generated from a distribution $Q$ with density 
\begin{equation*}
    \begin{aligned}
q(x)= 0.3 d_1(x; \mu_1, \sigma_1) + 0.5 d_2(x; \mu_2, \sigma_2) + 0.2 d_3(x), \quad x \in \mathcal{X}=[-10,10], 
    \end{aligned}
\end{equation*}
where $d_1$ and $d_2$ denote, respectively, the densities of truncated normal and truncated Laplace random variables; $d_3$ is the uniform density. The values of the location parameters are $\mu_1=-5$, $\mu_2 =5$ and scale parameters are $\sigma_1=\sigma_2= 3$. We aim to test the hypothesis in \eqref{eqn:h0h1} for three different specifications of the null density, i.e.,
\begin{equation*}
    \begin{aligned}
     g_{\vect{\beta},1}(x) & = 0.5 d_1(x;\beta_1, \beta_2) + 0.5 d_3(x); \\
     g_{\vect{\beta},2}(x) &= 0.3 d_1(x;-5,1) + 0.7 d_1(x;\beta_1,\beta_2);\\
     g_{\vect{\beta},3}(x) &= \beta_1 d_1(x; -4, 1) + \beta_2 d_2(x; 4, 1) +(1-\beta_1-\beta_2) d_3(x), 
\end{aligned}
\end{equation*}
where $\vect{\beta}=(\beta_1, \beta_2)$ is the unknown parameter vector to be estimated. 
Their corresponding CDFs are denoted by $G_{\vect{\beta},1}$, $G_{\vect{\beta},2}$, and $G_{\vect{\beta},3}$, respectively. The reference distribution, $F_{\vect{\gamma}}$, is chosen to be a truncated normal distribution over $\mathcal{X}$ with unknown parameter $\vect{\gamma}$ corresponding to its mean and variance. Figure~\ref{fig:histogram1} shows the histogram of the dataset considered, along with the densities of $q, g_{\vect{\gamma}}, g_{\vect{\beta},1}, g_{\vect{\beta},2},$ and $g_{\vect{\beta},3}$ estimated via maximum likelihood.

\begin{figure}[t]
    \centering
    \includegraphics[width=0.48\textwidth]{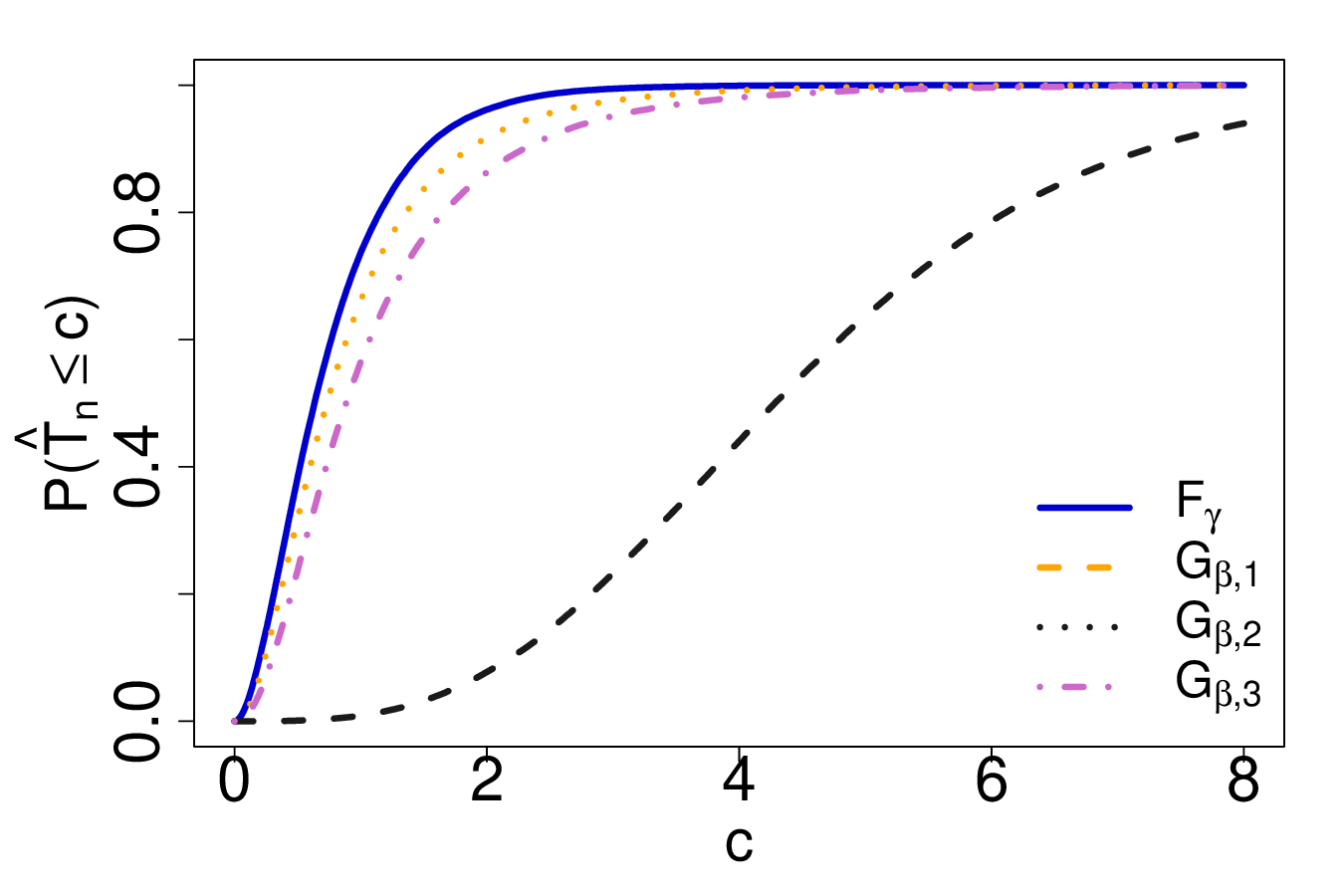}
    \includegraphics[width=0.48\textwidth]{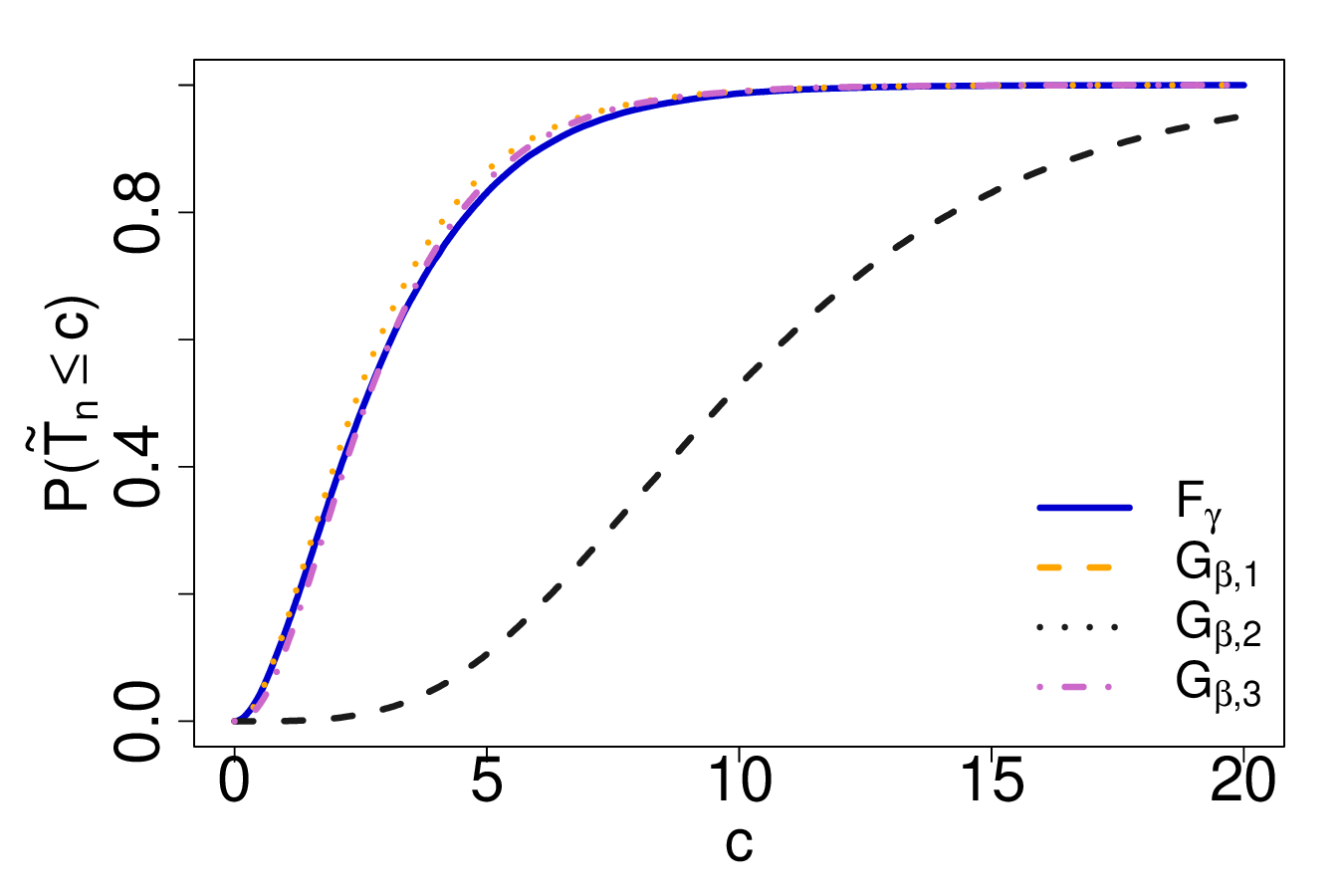} 
    \includegraphics[width=0.48\textwidth]{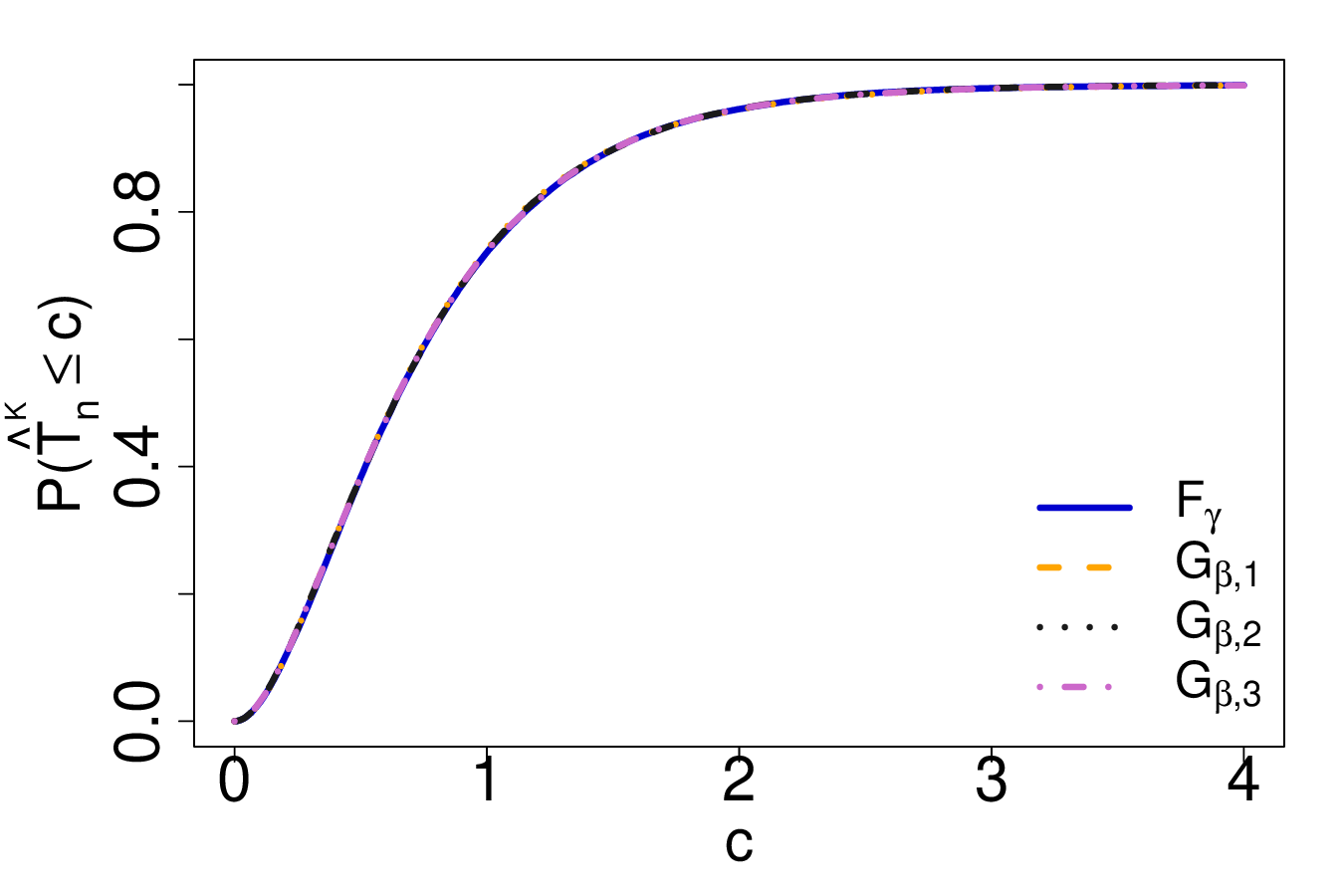} %
    \includegraphics[width=0.48\textwidth]{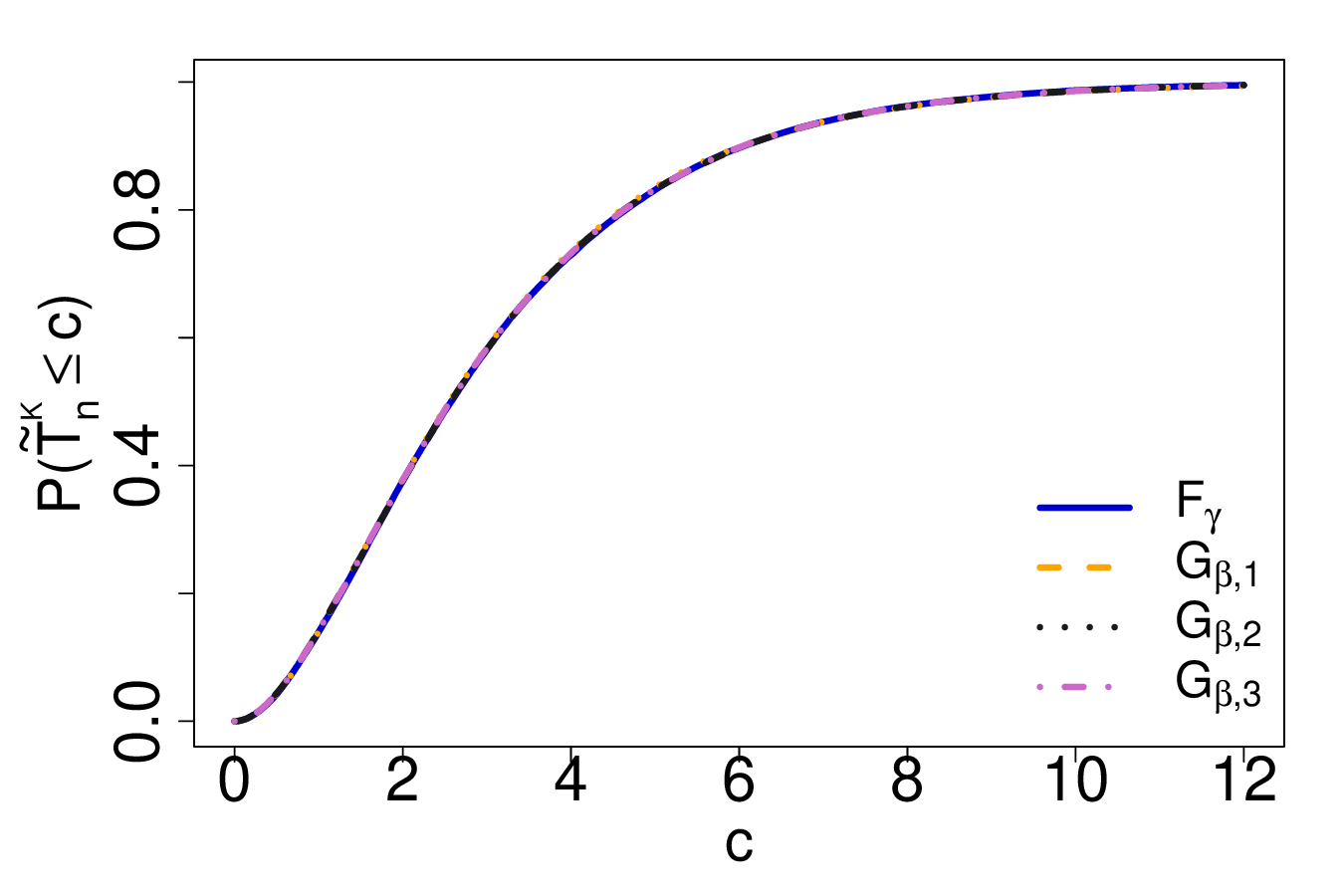}%
    \\[-0.7em]
    \caption{The simulated null distributions of the order selection statistics (left panels) and the subset selection test statistics (right panels), using basis functions obtained by composing the normalized shifted Legendre polynomials with the null CDFs (upper panels) and the K2 transform (lower panels), 
    under $F_{\vect{\gamma}}$, $G_{\vect{\beta},1}$, $G_{\vect{\beta},2}$, and $G_{\vect{\beta},3}$. }
\label{fig:image3}
\end{figure}

Consider the case in which the basis functions used in \eqref{eqn:alter} are constructed as compositions of the normalized shifted Legendre polynomials on [0,1] and the CDF under the null, i.e., 
\begin{equation}
\begin{split}
h_{1,\beta}(x)&=\sqrt{12}[G_{\vect{\beta},s}(x)-0.5],\\
h_{2,\beta}(x)&=\sqrt{5}[6G^2_{\vect{\beta},s}(x)-6G_{\vect{\beta},s}(x)+1],\\
h_{3,\beta}(x)&=\sqrt{7}[20G^3_{\vect{\beta},s}(x)-30G^2_{\vect{\beta},s}(x)+12G_{\vect{\beta},s}(x)-1],\\
h_{4,\beta}(x)&=3[70G^4_{\vect{\beta},s}(x)-140G^3_{\vect{\beta},s}  +90G^2_{\vect{\beta},s}(x)-20G_{\vect{\beta},s}(x)+1],\quad\text{etc.}
\label{eqn:legendre}
\end{split}
\end{equation}
with $s=1,2,3$ when testing $G_{\vect{\beta},1}$, $G_{\vect{\beta},2}$, and $G_{\vect{\beta},3}$, respectively. 
For each of the three null models considered,  we simulated the null distributions of the statistics $\widehat{T}_n$ in \eqref{eqn:ordersel} and $\widetilde{T}_n$ in \eqref{eqn:subsetsel} constructed using the above-mentioned sets of basis functions, up to $M_n=6$, and with $\widehat{S}_{m,n}$ chosen to be the unnormalized generalized score statistic in \eqref{eqn:non_norm}. The simulation was performed using the projected bootstrap, described in Section \ref{sec:3.1}, with 100,000 replicates. For comparison, the distribution of the statistics in \eqref{eqn:K2st2} under $F_{\vect{\gamma}}$ has also been simulated following the same sampling scheme and choosing the basis functions $\phi_j$ to be a composition of normalized shifted Legendre polynomials on $[0,1]$ and  $F_{\vect{\gamma}}$; hence, for $j=1,\dots,4$ they correspond to the right-hand sides of the equations in \eqref{eqn:legendre} with $F_{\vect{\gamma}}$ in place of $G_{\vect{\beta},s}$. The resulting simulated null distributions are shown in the upper panels of Figure~\ref{fig:image3}, where the blue solid curve corresponds to the statistics under $F_{\vect{\gamma}}$, and the yellow dotted, black dashed, and purple dash-dotted curves correspond to those constructed under  $G_{\vect{\beta},1}$, $G_{\vect{\beta},2}$, and $G_{\vect{\beta},3}$, respectively.
As shown in the upper-left panel of Figure~\ref{fig:image3}, the null distributions of the order selection statistic $\widehat{T}_n$ under $G_{\vect{\beta},1}$, $G_{\vect{\beta},2},$ and $G_{\vect{\beta},3}$ differs substantially from that constructed under $F_{\vect{\gamma}}$, and given by the first expression in \eqref{eqn:K2st2}. 
In the case of the subset selection test statistic $\widetilde{T}_n$, the null distributions under $G_{\vect{\beta},1}$, and $G_{\vect{\beta},3}$ are rather similar. They also match the null distribution of the same statistic, constructed under $F_{\vect{\gamma}}$, and given by the right expression in \eqref{eqn:K2st2}. However, all these distributions differ substantially from that of $\widetilde{T}_n$ under $G_{\vect{\beta},2}$. These discrepancies are expected given that, in general, the unnormalized generalized score statistic in \eqref{eqn:non_norm} is not distribution-free.

\begin{table*}[t]
\begin{center}
\small	
\setlength{\tabcolsep}{4pt}
\begin{tabular}{|c|cccc|cccc|cccc|}
\hline
\multirow{2}{*}{}
 & \multicolumn{4}{c|}{$\alpha=0.001$}
 & \multicolumn{4}{c|}{$\alpha=0.05$}
 & \multicolumn{4}{c|}{$\alpha=0.1$} \\ 
\cline{2-13}
 & \rule{0pt}{1.4em} $\widehat{T}_n$ & $\widetilde{T}_n$ & $\widehat{T}^{K}_n$ & $\widetilde{T}^{K}_n$
 & $\widehat{T}_n$ & $\widetilde{T}_n$ & $\widehat{T}^{K}_n$ & $\widetilde{T}^{K}_n$
 & $\widehat{T}_n$ & $\widetilde{T}_n$ & $\widehat{T}^{K}_n$ & $\widetilde{T}^{K}_n$ \\
\hline
$F_{\vect{\gamma}}$
 & 0.859 & 0.851 &  --  &  --
 & 1.000 & 1.000 &  --  &  --
 & 1.000 & 1.000 &  --  &  --  \\
\hline
$G_{\vect{\beta},1}$
 & 0.008 & 0.091 & 0.053 & 0.025
 & 0.351 & 0.551 & 0.519 & 0.333
 & 0.533 & 0.688 & 0.662 & 0.473 \\
\hline
$G_{\vect{\beta},2}$
 & 0.016 & 0.038 & 0.643 & 0.695
 & 0.238 & 0.424 & 0.967 & 0.969
 & 0.357 & 0.566 & 0.985 & 0.986 \\
\hline
$G_{\vect{\beta},3}$
 & 0.007 & 0.028 & 0.037 & 0.035
 & 0.164 & 0.285 & 0.371 & 0.351
 & 0.286 & 0.402 & 0.498 & 0.482 \\
\hline
\end{tabular}
\end{center}
\caption{Power comparison between classical order selection and subset selection test 
statistics with their K2-transformed counterparts. The null distributions are $G_{\vect{\beta},1}, G_{\vect{\beta},2}, G_{\vect{\beta},3}$; the reference distribution is $F_{\vect{\gamma}}$; the true data-generating distribution is $Q$; and significance levels are set at $0.001, 0.05,$ and $0.1$.}
\label{tab:1}
\end{table*}

To assess the impact of the K2 transform, we repeated the same experiment but for the case in which the basis functions used in \eqref{eqn:alter} when testing $G_{\vect{\beta},1}$, $G_{\vect{\beta},2}$, $G_{\vect{\beta},3}$ are the K2-transformed basis functions, obtained as described in Section~\ref{sec:4.1}, and corresponding to the images of the K2 transform applied to the basis functions $\phi_j$ given by compositions of Legendre polynomials and $F_{\vect{\gamma}}$. In what follows, we denote the order selection and subset selection statistics constructed using the K2-transformed basis functions by $\widehat{T}^{K}_n$ and $\widetilde{T}^{K}_n$, respectively. As shown in the lower panels of Figure~\ref{fig:image3}, for both such statistics, the simulated null distributions under $G_{\vect{\beta},1}$, $G_{\vect{\beta},2}$, and $G_{\vect{\beta},3}$ are  indistinguishable from the distribution of the corresponding statistics in \eqref{eqn:K2st2} simulated under $F_{\vect{\gamma}}$. It follows that, for all models and statistics considered in this example, a sample size of 100 observations is sufficient to retrieve the distribution-free property even when two parameters are estimated. 

To evaluate the performance of the testing procedure in \eqref{eqn:h0h1} under different choices of basis functions, we computed the simulated power using both order selection and subset selection statistics and calculated it as the proportion of bootstrap replications under $Q$ for which the value of these statistics  exceeds the critical value obtained from the simulated null distribution at the prescribed significance levels. 
Table~\ref{tab:1} presents the resulting power when testing $G_{\vect{\beta},1}, G_{\vect{\beta},2}, G_{\vect{\beta},3}$ using the statistics $\widehat{T}_n$, $\widetilde{T}_n$, $\widehat{T}^{K}_n$, and $\widetilde{T}^{K}_n$  at significance levels $\alpha=0.001, 0.05,$ and $0.1$. The power when testing the reference distribution $F_{\vect{\gamma}}$ using $\widehat{T}_n$ and $\widetilde{T}_n$ is also reported. When testing $G_{\vect{\beta}_2}$ and $G_{\vect{\beta}_3}$, the K2-based test statistics $\widehat{T}^{K}_n$ and $\widetilde{T}^{K}_n$ exhibit substantially higher power than $\widehat{T}_n$ and $\widetilde{T}_n$. Notably, when testing $G_{\vect{\beta},2}$ at a significance level of $\alpha=0.001$, $\widetilde{T}_n^{K}$ attains power as high as 0.695, whereas $\widetilde{T}_n$ achieves only 0.038. When testing $G_{\vect{\beta}_1}$, $\widehat{T}^{K}_n$ shows higher power than $\widehat{T}_n$ while the power of $\widetilde{T}^{K}_n$ is  lower than that of $\widetilde{T}_n$. This reflects the fact that K2-transformed statistics can exhibit higher power against certain alternatives, though such improvements in power should not be assumed to be true in general.  

\section{Case study: analyzing an X-ray spectrum from RT Cru}  \label{sec:5}
In X-ray astronomy, spectral analysis is essential for understanding the fundamental properties of stars, galaxies, and other celestial objects. In particular, the presence of spectral lines in X-ray spectra provides valuable insights into an object’s chemical composition, distance from Earth, temperature, motion, surrounding environments, and other important attributes.

Here, we focus on the study of a high-resolution spectrum from the star RT Cru and obtained in November 2015 by the \textit{Chandra} X-ray Observatory \citep{swartz}. RT Cru is of particular astronomical significance because it belongs to the rare class of X-ray-emitting symbiotic systems --  crucial for studying Type Ia supernovae\footnote{A Type Ia supernova is a powerful nuclear explosion that occurs when a small, dense star called a white dwarf gathers too much mass from a nearby companion star. Once it exceeds its stability limit, the white dwarf undergoes an uncontrollable burst of nuclear fusion.} and, more broadly, investigating the expansion of the Universe. 
In \cite{zhang2023}, smooth tests were primarily employed to assess the departure from uniformity in a background-only spectrum, that is, a spectrum in which no spectral lines are present. Likelihood ratio tests were then used to test for the presence of spectral lines in a spectrum that was known to contain at least three spectral lines in the wavelength region between 1.65 Å and 2.05 Å (where 1Å$= 10^{-10}$m). 
Focusing on the latter set of data, here we use the methodology described in Section \ref{sec:4} to assess the validity of the parametric model: 
\begin{equation}
    \begin{aligned}
g_{\vect{\beta}}(x) =  (1- \beta_{1}-\beta_{2}-\beta_{3}) b(x) + \sum\limits_{r=1}^3 \beta_r s_{r}(x), \quad x\in  \mathcal{X},
    \label{eqn:trueg}
    \end{aligned}    
\end{equation}
where $\mathcal{X}=[1.65,2.05]$; $b(x)$ is a uniform background density on $\mathcal{X}$; and the functions $s_r(x)$ model each of the the three expected spectral lines. They consist of a convolution of a normal density with a Moffat function\footnote{The Moffat function is equivalent to the density of a location-scale Student’s $t$-distribution with a location of 0, a scale of 0.025, and 4 degrees of freedom. } \citep[][]{1969moffat} and specify as 
\begin{equation}
   \begin{aligned}
s_r\left(x\right) \propto \int_{-\infty}^{+\infty}  \frac{\exp \left\{-\frac{1}{2}\left(\frac{w_0-\mu_r}{\sigma_r}\right)^2\right\}}{\left[1+\left(\frac{x-w_0}{0.05}\right)^2\right]^{2.5}} d w_0, \quad x\in\mathcal{X},
   \end{aligned}    
   \label{eq:LRF}
\end{equation}
where $\mu_r, \sigma_r$ are known parameters, with $\mu_1 = 1.78499, \mu_2 = 1.85247$, and $\mu_3 = 1.94365$, and $\sigma_1=\sigma_2=\sigma_3 = 0.0025$. The unknown parameters $\beta_{r}$ represent the relative intensities of the spectral lines and are estimated via maximum likelihood.

\begin{figure}%
    \centering
    \includegraphics[width=0.48\textwidth]{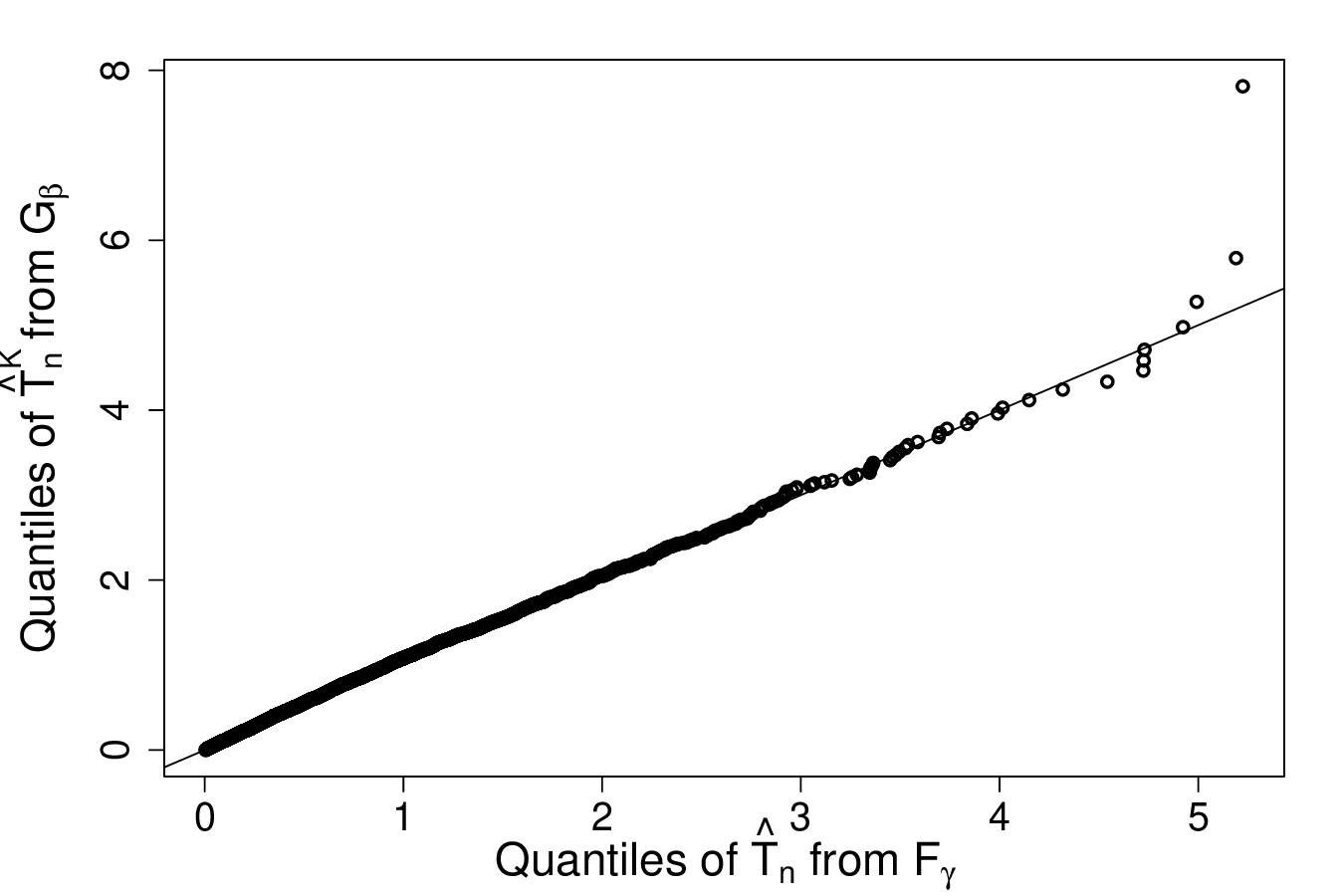} %
    \includegraphics[width=0.48\textwidth]{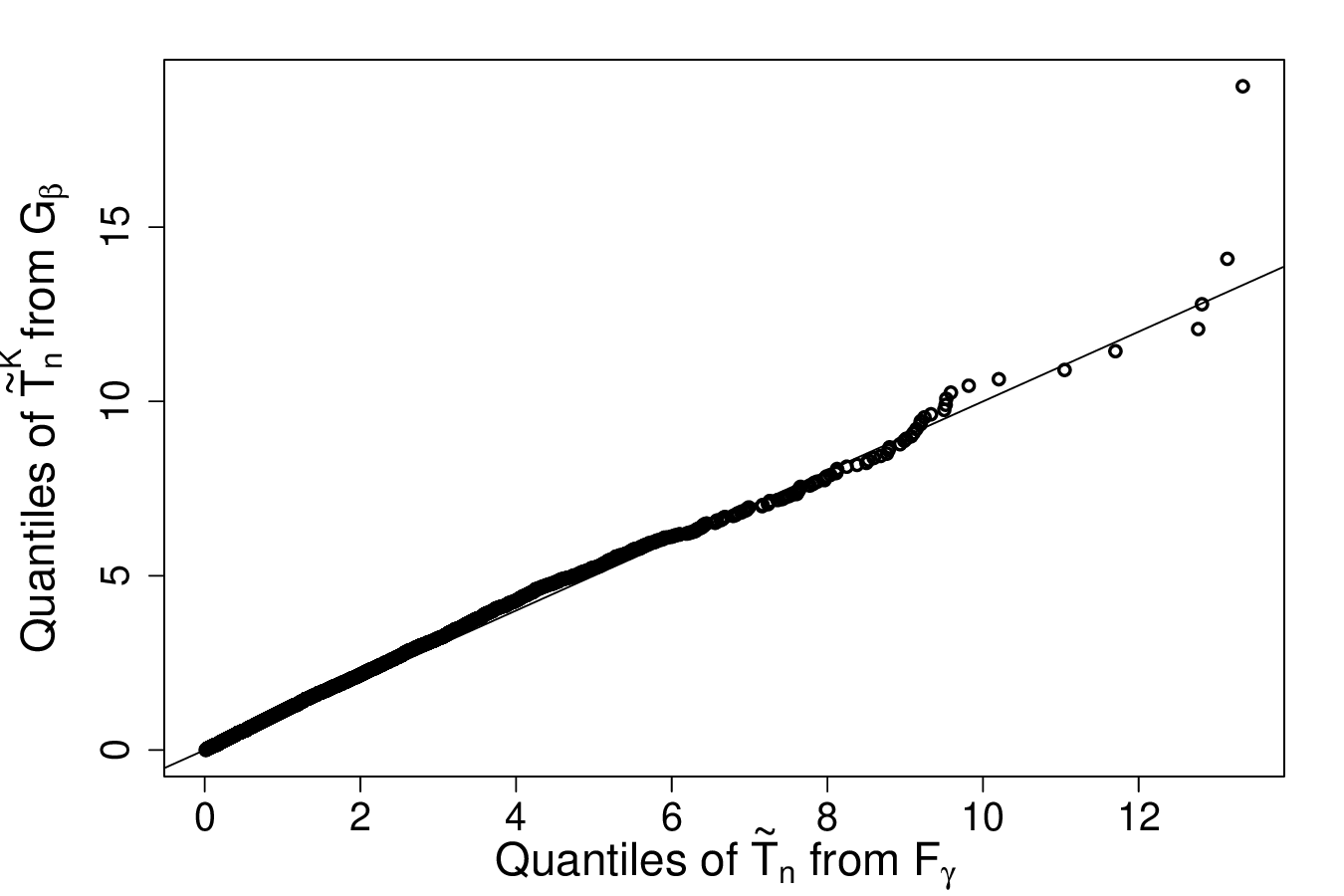}%
    \caption{Left: QQ plots of the simulated order selection test statistics in \eqref{eqn:ordersel} under $G_{\vect{\beta}}$ and in \eqref{eqn:K2st2} under $F_{\vect{\gamma}}$. Right: QQ plots of the simulated subset selection test statistics in \eqref{eqn:subsetsel} 
under $G_{\vect{\beta}}$ and in \eqref{eqn:K2st2} under $F_{\vect{\gamma}}$. For both statistics computed under $G_{\vect{\beta}}$,  the functions $h_{j\vect{\beta}}$ are elements of the $K2$ basis costructed as in Proposition \ref{prop:6}.}
 \label{fig:image4}
\end{figure}

We use the approach described in Section \ref{sec:4} to test the validity of \eqref{eqn:trueg}. The reference distribution $F_{\vect{\gamma}}$ considered is a simplified version of \eqref{eqn:trueg}. In particular, it consists of a convex combination of the uniform distribution and truncated normal distributions over $\mathcal{X}$. Its density is:
\begin{equation*}
    \begin{aligned}
        f_{\vect{\gamma}}(x) = (1- \gamma_{1}-\gamma_{2}-\gamma_{3}) b(x) + \sum\limits_{r=1}^3 \gamma_r p_{r}(x,\mu_r,0.05), \quad x\in\mathcal{X}
    \end{aligned}
\end{equation*}
where $p_r$ is the density of a truncated normal over $\mathcal{X}$, with $\mu_r$ being the known positions of the spectral lines and listed after \eqref{eq:LRF}. The unknown parameter $\vect{\gamma}=(\gamma_1,\gamma_2,\gamma_3)$ is estimated via the maximum likelihood. The basis functions for $F_{\vect{\gamma}}$ are chosen to be compositions of normalized shifted Legendre polynomials on $[0,1]$ with $F_{\vect{\gamma}}$. These functions are employed as a starting point to construct the K2-transformed basis functions $h_{j\vect{\beta}}$ for $L^2(G_{\vect{\beta}})$, whose residuals are subsequently used when calculating the order selection and subset selection statistics in \eqref{eqn:ordersel} and \eqref{eqn:subsetsel}, respectively, with $\widehat{S}_{m,n}$ as in \eqref{eqn:non_norm} and $\{h_{j\vect{\beta}}\}_{j=1}^{M_n}$ as in Proposition \ref{prop:6}. Hence, consistently with the notation used in the previous section, we denote them with $\widehat{T}_n^K$ and $\widetilde{T}_n^K$. Their limiting null distributions are obtained by simulating 100,000 realizations of the statistics \eqref{eqn:K2st2} under $F_{\bm \gamma}$ through the projected bootstrap (cf. Section \ref{sec:3}). The maximum number of basis functions used is  $M_n=6$. To assess the accuracy of the approximation, we also simulated the distributions of $\widehat{T}_n^K$ and $\widetilde{T}_n^K$ under $G_{\vect{\beta}}$. As shown by the QQ-plots in Figure \ref{fig:image4}, the null distributions of such statistics
under $G_{\vect{\beta}}$ are indistinguishable from those of the statistics in \eqref{eqn:K2st2} under  $F_{\vect{\gamma}}$. 

Finally, the p-values from the order selection and subset selection test statistics are 0.463 and 0.454, respectively, indicating that the hypothesized model in \eqref{eqn:trueg} fits the observed spectrum well. 

From an astrophysical standpoint, the Gaussian peaks $s_r(x)$, $r=1,2,3$, in \eqref{eqn:trueg} correspond to iron lines\footnote{Each chemical element produces a unique set of spectral lines that arise from changes in atomic energy levels. Iron stands out because it has an exceptionally large number of these lines spanning ultraviolet, visible, and infrared light. As a result, iron is extremely valuable for modeling and interpreting observed spectra in many scientific studies.} from various ionization states. Specifically, $s_1$ and $s_2$ correspond to the Fe XXV and Fe XXVI lines, which occur in extremely hot conditions where the iron has lost most of its electrons. In contrast, $s_3$ corresponds to the Fe K$\alpha$ line -- a fluorescent signal emitted by iron atoms that still hold most of their electrons, indicating that some X-rays  
are being reflected by nearby, cooler, denser material. Hence, by failing to reject the model in  \eqref{eqn:trueg}, our test is in agreement with the claim of \cite{Danehkar2021} of a multi-phase environment in \rtcru, where a very hot, highly ionized plasma -- responsible for the Fe XXV and Fe XXVI lines -- coexists with cooler, denser material that produces the Fe K$\alpha$ fluorescence. 

\section{Summary and discussion}
\label{sec:6}
This article introduces a new class of smooth test statistics whose asymptotic distribution-freeness is achieved by relying on the K2 transform. The latter consists of a change of variable in functional space, which enables the construction of an empirical process with a standard asymptotic null distribution.  In the context of smooth tests, such a transformation is especially valuable in that it yields a new family of orthonormal bases in $L^2(G_{\beta})$. Test statistics based on elements of such bases are shown to be asymptotically distribution-free even when the parameters are estimated, and the sample size is only moderately large.

The projected bootstrap is also discussed as a computationally efficient alternative to the classical parametric bootstrap. In particular,  the projection structure induced by parameter estimation allows us to simulate the null distribution of the test statistics of interest without re-estimating the model parameters at each bootstrap replicate. Simulation experiments show that the computational gain attained by the projected bootstrap can be substantial compared to the parametric bootstrap, especially when the estimation of the parameters is CPU-intensive.

While the present manuscript focuses on the univariate setting, the proposed framework can be easily adapted to test multivariate parametric models. In particular, when $G_{\bm{\beta}}$ is $D$-dimensional distribution, the reference distribution, $F_{\bm{\gamma}}$ could be chosen to be a product of $D$ univariate distributions $F_{\bm{\gamma},d}$, $d=1,\dots,D$. The K2 orthonormal basis in $L^2(G_{\bm{\beta}})$ can then be constructed by applying the K2 transform to a tensor product of bases in $L^2(F_{\gamma,d})$ (e.g., \cite{algeri2021}). 

Extensions to smooth tests for regression \citep{inglot2006data, rayner2022smooth} are also possible. In this case, instead of relying on the classical empirical process for i.i.d. data to express the statistics of interest as functionals from it (see Section \ref{sec:2}),  the random measure to be used is the weighted empirical process proposed in  \cite{KHMALADZE2017348}.

\section{Acknowledgments}
The authors thank Charlie Geyer, Teresa Ledwina, and two anonymous referees for the valuable feedback.

The work of XZ and SA was in part supported by the Research and Innovation Office at the University of Minnesota. SA was partially supported by NSF grant DMS-2152746. XZ was partially supported by the University of Minnesota Data Science Initiative.

\section{Code Availability}
The R code used to conduct the simulations and analyses in Sections \ref{sec:3.1}, \ref{sec:4.2}, and \ref{sec:5} is available at \url{https://github.com/xiangyu2022/DisfreeSmoothTests}.

\appendix
\section*{Appendix} 

\mysection{Proof of Proposition 4}  \label{App:unitary}
\begin{proof}
To show that $K$ is a unitary operator, we need to demonstrate that it is surjective and preserves the inner product. It is surjective because for any function in $L^{2}(G_{\vect{\beta}})$, if $\phi \perp span(\vect{1},l_{\vect{\gamma},\vect{\beta}})$,  
\begin{equation*}
\begin{aligned}
K\phi &= \phi-\frac{l_{\vect{\gamma},\vect{\beta}}-\vect{1}}{\vect{1}-\langle l_{\vect{\gamma},\vect{\beta}}, \vect{1}\rangle_{G_{\vect{\beta}}}}\langle l_{\vect{\gamma},\vect{\beta}}-\vect{1}, \phi \rangle_{G_{\vect{\beta}}} = \phi.
\end{aligned}
\end{equation*}
Otherwise if $\phi = c_1 \vect{1} + c_2 l_{\vect{\gamma},\vect{\beta}},$ we have 
\begin{equation}
\begin{aligned}
K(c_1 l_{\vect{\gamma},\vect{\beta}}+ c_2 \vect{1}) &= c_1 l_{\vect{\gamma},\vect{\beta}}+ c_2 \vect{1} - \frac{l_{\vect{\gamma},\vect{\beta}}-\vect{1}}{\vect{1}-\langle l_{\vect{\gamma},\vect{\beta}}, \vect{1}\rangle_{F_{\vect{\gamma}}}}\langle l_{\vect{\gamma},\vect{\beta}}-\vect{1}, c_1 l_{\vect{\gamma},\vect{\beta}}+ c_2 \vect{1} \rangle_{F_{\vect{\gamma}}} \\
&= c_1 l_{\vect{\gamma},\vect{\beta}}+ c_2 \vect{1} - c_1 (l_{\vect{\gamma},\vect{\beta}}-\vect{1}) +c_2 (l_{\vect{\gamma},\vect{\beta}}- \vect{1}) \\ 
&= c_1 \vect{1} + c_2 l_{\vect{\gamma},\vect{\beta}}.
\label{eqn:c1c2}
\end{aligned}
\end{equation}
Notice here, by letting $c_1=1, c_2=0$ or $c_1=0, c_2 = 1$ in equation \eqref{eqn:c1c2}, we obtain 
$$K l_{\vect{\gamma},\vect{\beta}} = \vect{1}, \quad K \vect{1} = l_{\vect{\gamma},\vect{\beta}}.$$
The operator $K$ preserves the inner product because for any functions $\phi_1, \phi_2 \in L^2(G_{\vect{\beta}})$ 
\begin{equation*}
\begin{aligned}
&\inner{K\phi_1,K\phi_2}_{G_{\vect{\beta}}} \\
&= \inner{\phi_1,\phi_2}_{G_{\vect{\beta}}} - \frac{2\inner{\vect{1}-l_{\vect{\gamma},\vect{\beta}},\phi_1}_{G_{\vect{\beta}}}\inner{\vect{1}-l_{\vect{\gamma},\vect{\beta}},\phi_2}_{G_{\vect{\beta}}}}{1-\inner{l_{\vect{\gamma},\vect{\beta}},1}_{G_{\vect{\beta}}}} \\
& \hspace{2.2cm}+ \frac{\inner{\vect{1}-l_{\vect{\gamma},\vect{\beta}},\phi_1}_{G_{\vect{\beta}}}\inner{\vect{1}-l_{\vect{\gamma},\vect{\beta}},\phi_2}_{G_{\vect{\beta}}} \inner{\vect{1}-l_{\vect{\gamma},\vect{\beta}},\vect{1}-l_{\vect{\gamma},\vect{\beta}}}_{G_{\vect{\beta}}}}{(1-\inner{l_{\vect{\gamma},\vect{\beta}},1}_{G_{\vect{\beta}}})^2} \\ 
&= \inner{\phi_1,\phi_2}_{G_{\vect{\beta}}} + \frac{\inner{\vect{1}-l_{\vect{\gamma},\vect{\beta}},\phi_1}_{G_{\vect{\beta}}}\inner{\vect{1}-l_{\vect{\gamma},\vect{\beta}},\phi_2}_{G_{\vect{\beta}}} }{(1-\inner{l_{\vect{\gamma},\vect{\beta}},1}_{G_{\vect{\beta}}})^2} \left(- 2 + 2 \inner{l_{\vect{\gamma},\vect{\beta}},1}_{G_{\vect{\beta}}} + \inner{\vect{1}-l_{\vect{\gamma},\vect{\beta}},\vect{1}-l_{\vect{\gamma},\vect{\beta}}}_{G_{\vect{\beta}}} \right)\\ 
&= \inner{\phi_1,\phi_2}_{G_{\vect{\beta}}}. 
\end{aligned}
\end{equation*}
$K$ is self-adjoint because 
\begin{equation*}
    \begin{aligned}
\inner{K\phi_1, \phi_2}_{G_{\vect{\beta}}} &= \inner{\phi_1-\frac{l_{\vect{\gamma},\vect{\beta}}-\vect{1}}{\vect{1}-\langle l_{\vect{\gamma},\vect{\beta}}, \vect{1}\rangle_{G_{\vect{\beta}}}}\langle l_{\vect{\gamma},\vect{\beta}}-\vect{1}, \phi_1 \rangle_{G_{\vect{\beta}}},\phi_2}_{G_{\vect{\beta}}} \\ 
&= \inner{\phi_1, \phi_2}_{G_{\vect{\beta}}} - \frac{\inner{l_{\vect{\gamma},\vect{\beta}}-\vect{1},\phi_1}_{G_{\vect{\beta}}}\inner{l_{\vect{\gamma},\vect{\beta}}-\vect{1},\phi_2}_{G_{\vect{\beta}}}}{\vect{1}-\langle l_{\vect{\gamma},\vect{\beta}}, \vect{1}\rangle_{G_{\vect{\beta}}}}\\
&= \inner{\phi_1,\phi_2-\frac{l_{\vect{\gamma},\vect{\beta}}-\vect{1}}{\vect{1}-\langle l_{\vect{\gamma},\vect{\beta}}, \vect{1}\rangle_{G_{\vect{\beta}}}}\langle l_{\vect{\gamma},\vect{\beta}}-\vect{1}, \phi_2 \rangle_{F_{\vect{\gamma}}}}_{G_{\vect{\beta}}} \\ 
&= \inner{\phi_1, K\phi_2}_{G_{\vect{\beta}}}.
    \end{aligned}
\end{equation*}
The unitary and self-adjoint properties of $K$ imply $K^2=I.$
\end{proof}

\mysection{Required Orthogonality Conditions of \texorpdfstring{$\widetilde{c}_{\lambda_k}$}{TEXT}}  \label{App:orthor}
\begin{proof}
It can be verified that 
$$ \inner{
\widetilde{c}_{\lambda_{2}},\vect{1}}_{G_{\vect{\beta}}}= \inner{
U_{b_{\vect{\beta}_{1}}c_{\lambda_{1}}}
c_{\lambda_{2}},\vect{1}}_{G_{\vect{\beta}}} = \inner{c_{\lambda_{2}},U_{b_{\vect{\beta}_{1}}c_{\lambda_{1}}}\vect{1}}_{G_{\vect{\beta}}}=\inner{c_{\lambda_{2}},\vect{1}}_{G_{\vect{\beta}}}=0,$$
$$
\inner{
\widetilde{c}_{\lambda_{2}},b_{\vect{\beta}_1}}_{G_{\vect{\beta}}}= \inner{
U_{b_{\vect{\beta}_{1}}c_{\lambda_{1}}}
c_{\lambda_{2}},b_{\vect{\beta}_1}}_{G_{\vect{\beta}}} = \inner{c_{\lambda_{2}},U_{b_{\vect{\beta}_{1}}c_{\lambda_{1}}}b_{\vect{\beta}_1}}_{G_{\vect{\beta}}}=\inner{c_{\lambda_{2}},c_{\lambda_1}}_{G_{\vect{\beta}}}=0;$$
and
\begin{equation*}
\begin{split}
 \inner{
\widetilde{c}_{\lambda_{3}},\vect{1}}_{G_{\vect{\beta}}}&= \inner{
U_{b_{\vect{\beta}_{2}}\widetilde{c}_{\lambda_{2}}}\circ U_{b_{\vect{\beta}_{1}}c_{\lambda_{1}}}
c_{\lambda_{3}},\vect{1}}_{G_{\vect{\beta}}} = \inner{c_{\lambda_{3}},U_{b_{\vect{\beta}_{1}}c_{\lambda_{1}}}\circ U_{b_{\vect{\beta}_{2}}\widetilde{c}_{\lambda_{2}}}\vect{1}}_{G_{\vect{\beta}}}=\inner{c_{\lambda_{3}},\vect{1}}_{G_{\vect{\beta}}}=0,\\
\inner{
\widetilde{c}_{\lambda_{3}},b_{\vect{\beta}_{1}}}_{G_{\vect{\beta}}}&= \inner{U_{b_{\vect{\beta}_{2}}\widetilde{c}_{\lambda_{2}}}\circ U_{b_{\vect{\beta}_{1}}c_{\lambda_{1}}}c_{\lambda_{3}},b_{\vect{\beta}_1}}_{G_{\vect{\beta}}} = \inner{c_{\lambda_{3}},U_{b_{\vect{\beta}_{1}}c_{\lambda_{1}}}\circ U_{b_{\vect{\beta}_{2}}\widetilde{c}_{\lambda_{2}}}b_{\vect{\beta}_1}}_{G_{\vect{\beta}}}=\inner{c_{\lambda_{3}},c_{\lambda_1}}_{G_{\vect{\beta}}}=0,\\
\inner{
\widetilde{c}_{\lambda_{3}},b_{\vect{\beta}_{2}}}_{G_{\vect{\beta}}} &= \inner{U_{b_{\vect{\beta}_{2}}\widetilde{c}_{\lambda_{2}}}\circ U_{b_{\vect{\beta}_{1}}c_{\lambda_{1}}}c_{\lambda_{3}},b_{\vect{\beta}_2}}_{G_{\vect{\beta}}} = \inner{c_{\lambda_{3}},U_{b_{\vect{\beta}_{1}}c_{\lambda_{1}}}\circ U_{b_{\vect{\beta}_{2}}\widetilde{c}_{\lambda_{2}}}b_{\vect{\beta}_2}}_{G_{\vect{\beta}}}\\&
=\inner{c_{\lambda_{3}}, U_{b_{\vect{\beta}_{1}},c_{\lambda_1}}\circ U_{b_{\vect{\beta}_{1}}c_{\lambda_{1}}}c_{\lambda_1}}_{G_{\vect{\beta}}}=0.
\end{split}
\end{equation*}
By induction, the same can be done for $\widetilde{c}_{\lambda_4},\dots,\widetilde{c}_{\lambda_p}$.

\end{proof}

\newpage
\nocite*{}  
\bibliographystyle{apalike}
\bibliography{bibfile}
\end{document}